\documentclass[a4paper,11pt]{amsart}
\usepackage[english]{babel}
\usepackage{amsmath}
\usepackage{amssymb}
\usepackage{mathrsfs}
\usepackage{enumerate}
\usepackage{mathtools}
\usepackage{tikz,tkz-euclide}
\usepackage{pgfplots}
\usepackage{hyperref}
\usetikzlibrary{arrows}
\usepackage{makecell}

\usepackage[most]{tcolorbox}
\usepackage{lmodern} 

\usepackage[]{algorithm2e}

\newtheorem{thm}{Theorem}[section]
\newtheorem{Lemma}[thm]{Lemma}
\newtheorem{Proposition}[thm]{Proposition}

\newtheorem*{thm*}{Theorem}

\theoremstyle{definition}

\newtheorem{Definition}[thm]{Definition}
\newtheorem{Remark}[thm]{Remark}

\newtheorem{Alg}[thm]{Algorithm}

\definecolor{wwwwww}{rgb}{0.4,0.4,0.4}

\DeclareMathOperator{\codim}{codim}

\DeclareMathOperator{\rk}{rk}

\DeclareMathOperator{\Hilb}{Hilb}

\DeclareMathOperator{\Sym}{Sym}

\DeclareMathOperator{\Sec}{\mathbb{S}ec}

\DeclareMathOperator{\VSP}{VSP}

\hypersetup{pdfpagemode=UseNone}
\hypersetup{pdfstartview=FitH}
		
\begin{document}

\title{Decomposition algorithms for tensors and polynomials}

\author[Antonio Laface]{Antonio Laface}
\address{\sc Antonio Laface\\
Departamento de Matematica, Universidad de Concepci\'on\\
Casilla 160-C, Concepci\'on\\
Chile}
\email{alaface@udec.cl}

\author[Alex Massarenti]{Alex Massarenti}
\address{\sc Alex Massarenti\\ Dipartimento di Matematica e Informatica, Universit\`a di Ferrara, Via Machiavelli 30, 44121 Ferrara, Italy}
\email{alex.massarenti@unife.it}

\author[Rick Rischter]{Rick Rischter}
\address{\sc Rick Rischter\\
Universidade Federal de Itajub\'a (UNIFEI)\\ 
Av. BPS 1303, Bairro Pinheirinho\\ 
37500-903, Itajub\'a, Minas Gerais\\ 
Brazil}
\email{rischter@unifei.edu.br}

\date{\today}
\subjclass[2010]{Primary 14N07; Secondary 14N05, 51N35, 14Q15, 14N15}
\keywords{Tensor decomposition, Identifiability, Segre-Veronese, Secant varieties}

\begin{abstract}
We give algorithms to compute decompositions of a given polynomial, or more generally mixed tensor, as sum of rank one tensors, and to establish whether such a decomposition is unique. In particular, we present methods to compute the decomposition of a general plane quintic in seven powers, and of a general space cubic in five powers; the two decompositions of a general plane sextic of rank nine, and the five decompositions of a general plane septic. Furthermore, we give Magma implementations of all our algorithms.   
\end{abstract}

\maketitle
\setcounter{tocdepth}{1}
\tableofcontents

\section{Introduction}
Let $T$ be a tensor in a given tensor space over a field $K$, and consider additive decompositions of the form  
\stepcounter{thm}
\begin{equation}\label{eq1gen}
T = \lambda_1 U_1+...+\lambda_h U_{h}
\end{equation}
where the $U_i$'s are linearly independent rank one tensors, and $\lambda_i\in K^*$. The \textit{rank} of $T$, denoted by  $\rk(T)$, is the minimal positive integer $h$ such that $T$ admits a decomposition as in (\ref{eq1gen}). 

Tensor decomposition problems and techniques are of relevance in both pure and applied mathematics. For instance, tensor decomposition algorithms have applications in psycho-metrics, chemometrics, signal processing, numerical linear algebra, computer vision, numerical analysis, neuroscience and graph analysis \cite{BK09}, \cite{CM96}, \cite{CGLM08}, \cite{LO15}, \cite{MR13}.

We say that a tensor rank-1 decomposition has the \textit{generic identifiability property} if the expression (\ref{eq1gen}) is unique, up to permutations and scaling of the factors, on a dense open subset of the set of tensors admitting such an expression. Given a tensor rank-1 decomposition of length $h$ as in (\ref{eq1gen}) the problem of \textit{specific identifiability} consists in proving that such a decomposition is unique. Following \cite{COV17} we call an algorithm for specific identifiability \textit{effective} if it is sufficient to prove identifiability on a dense open subset of the set of tensors admitting a decomposition as in (\ref{eq1gen}). Therefore, an algorithm is effective if its constraints are satisfied generically, in other words if the same algorithm proves generic identifiability as well.

Our aim is to give efficient algorithms to explicitly compute a decomposition as in (\ref{eq1gen}) and to establish whether it is unique. The literature on this subjects is quite vast \cite{CM96}, \cite{BCMT10}, \cite{BB12}, \cite{OO13}, \cite{COV17}, \cite{Ba19}, \cite{AC20}, \cite{BT20}, \cite{MO20}. The majority of the available algorithms are based on the notion of eigenvector of a tensor and others deal with tensors of small rank. In this paper we propose a different approach that can be seen as a generalization of the classical catalecticant method.  

In Section \ref{Cat} we revise classical methods based on catalecticants and more generally on flattenings. In Section \ref{GCM} we introduce our main method for symmetric tensors which we then generalize to the case of mixed tensors in Section \ref{Mix}. We explain our main idea in the case of symmetric tensors. Let $F\in K[x_0,\dots,x_n]_d$ be a symmetric tensor, that is a homogeneous polynomial, and let $H_{\partial F}^s$ be the linear subspace of $\mathbb{P}(K[x_0,\dots,x_n]_{d-s})$ spanned by the partial derivatives of order $s$ of $F$. The catalecticant method basically consists in intersecting $H_{\partial F}^s$ with the Veronese variety $\mathcal{V}_{d-s}^n$ parametrizing powers of linear forms. Indeed, if $F$ admits a decomposition as sum of powers  
\stepcounter{thm}
\begin{equation}\label{eq2gen}
F = \lambda_1 L_1^d+...+\lambda_h L_{h}^d
\end{equation}
with $L_i\in K[x_0,\dots,x_n]_1$ then all its partial derivatives can be decomposed using the same linear forms. When $H_{\partial F}^s$ fills the span $\left\langle L_1^{d-s},\dots,L_{h}^{d-s}\right\rangle$ the linear forms $L_1,\dots,L_h$ can be recovered from the intersection $H_{\partial F}^s\cap \mathcal{V}_{d-s}^n$.

The main novelty in our method is that instead of intersecting with $\mathcal{V}_{d-s}^n$ we consider the intersection $H_{\partial F}^s\cap\Sec_{h-N_s}(\mathcal{V}_{d-s}^n)$, where $N_s = \binom{n+s}{s}-1$, and $\Sec_{h-N_s}(\mathcal{V}_{d-s}^n)$ is the $(h-N_s)$-secant variety of $\mathcal{V}_{d-s}^n$. Indeed, we prove that the decomposition of $F$ in (\ref{eq2gen}) can be reconstructed from such intersection. For instance, when $s =1$ our method works under the following bound
$$h < B_{n,d} := \frac{\binom{d-1+n}{n}+n^2}{n+1}.$$
The catalecticant method works at its best for even degree $d = 2k$ under the bound $h\leq \binom{n+k}{k}$. Note that this binomial coefficients is in general much smaller than $B_{n,d}$. The main drawback of our approach is that equations for secant varieties of Veronese varieties are known in very few cases \cite{LO13}. However, in the cases we were able to check empirically it turned out that the equations for secant varieties coming from classical flattenings are enough in order to establish whether a decomposition is unique, and in case to explicitly compute it. Furthermore, for degree two Veronese varieties these equations are classically known. So, our technique is very effective for computing decompositions of cubics and more generally ternary tensors. Indeed, out of this method we get an identifiability criterion for cubics for 
\stepcounter{thm}
\begin{equation}\label{bcub}
h < \frac{4n-\sqrt{8n+1}+3}{2}
\end{equation}
which also allows us to explicitly compute the decomposition. Furthermore, we prove that such criterion is effective for $h\leq n+2$. In particular, when $(n,h) = (3,5)$ we prove that the equality in (\ref{bcub}) is allowed, and we get a method to compute the decomposition in the Sylvester's pentahedral theorem \cite{Sy04}.

There are just other two cases in addition to the Sylvester's pentahedral theorem where a general polynomial $F\in \mathbb{C}[x_0,\dots,x_n]_d$ is $h$-identifiable, namely for $n = 1, d = 2m+1,h = m$ and $n=2,d=5,h=7$ \cite[Theorem 1]{GM19}. The first proof of the uniqueness of the decomposition  in seven powers of a general plane quintic is due to D. Hilbert \cite{Hi88}. This interesting case is not among the ones covered by our main method. However, in Section \ref{sec:Hilb} we introduce another technique, bases on linear projections from spaces spanned by derivatives, which turns out to be effective in Hilbert's case. Furthermore, as a variation of this method we give an algorithm to compute the five decompositions of a general plane septic in twelve powers \cite{Di07}. Moreover as a combination of our main method in Section \ref{GCM} and the concept of star configuration we get a faster algorithm that works when the given polynomial can be decomposed using linear forms defined over $\mathbb{Q}$.

In Section \ref{sub_gen} we consider the subgeneric cases. When $h$ is smaller than the generic rank we have that a general polynomial $F\in K[x_0,\dots,x_n]_d$ of rank $h$ is identifiable except when $(n,d,h)\in\{(2,6,9),(3,4,8),(5,3,9)\}$ and in these three cases there are exactly two decompositions \cite[Theorem 1.1]{COV18} which are contained in an elliptic curve. By looking at polynomials of small degree in the ideal of certain projections of these elliptic curves or of the relevant Veronese varieties we manage to give non trivial constraints that the decompositions must satisfy. In particular, when $(n,d,h) = (2,6,9)$ we produce an algorithm that successfully computes the two decompositions of a general plane sextic of rank nine.  

In Section \ref{vsp}, plugging-in the concept of variety of sums of powers, we consider the cases when a homogeneous polynomial admits infinitely many decompositions in $h$ powers. For instance, we successfully apply this method to plane quartics for $h = 6$, and to plane sextics for $h = 10$.

We implemented all our algorithms in Magma \cite{Magma97}. In the following table we list some cases in which our scripts managed to compute the decompositions:
\begin{center}
\begin{tabular}{l|c|c|c}
\textit{n} & \textit{d} & \textit{h} & Algorithms \\ 
\hline 
2 & 5 & $\leq 7$ & \ref{cat_alg}, \ref{Alg_Hilb} \\ 
\hline 
2 & 6 & $\leq 10$ & \ref{cat_alg}, \ref{hyp} and Remark \ref{rem_VSP} \\ 
\hline 
2 & 7 & $\leq 12$ & \ref{cat_alg}, \ref{alg_dix}\\  
\hline 
3 & 3 & $\leq 5$ & \ref{cat_g}\\
\hline 
6 & 3 & $\leq 9$ & \ref{cat_g}
\end{tabular} 
\end{center}
In Section \ref{Mix} we extend our main method in Section \ref{GCM} to mixed tensors. For instance, for Segre products of type $\mathbb{P}^n\times\mathbb{P}^n\times\mathbb{P}^n$ we get an algorithm that works for $h < 2n-\sqrt{n}+1$ while the classical flattening method in Section \ref{Cat} works for $h\leq n+1$. 

For example, Algorithm \ref{alg_S} successfully computed the decomposition of rank $h = 6$ tensors in $K^5\otimes K^{5}\otimes K^5$, and of rank $h = 11$ tensor in $K^{8}\otimes \Sym^2 K^{8}$.

Finally, we would like to stress that, since to establish identifiability it is enough to compute the degree of a $0$-dimensional scheme, our algorithms perform much better when just asked to determine whether a tensor is identifiable. For instance, Algorithms \ref{gen2}, \ref{alg_S} succeeded in establishing identifiability of rank $h = 15$ polynomials of degree $d = 3$ in $n+1 = 10$ variables, tensors of rank $h = 14$ in $K^{9}\otimes K^{9}\otimes K^{9}$, and tensors of rank $h = 14$ in $K^{9}\otimes \Sym^2K^{9}$. 

\subsection*{Organization of the paper} The paper is organized as follows. 
In Section \ref{Cat} we introduce the notation and recall the classical catalecticant method.
In Section \ref{GCM} we develop our main method for computing polynomial decompositions. In Section \ref{sec:Hilb} we introduce techniques, bases on linear projections and star configurations, to deal with plane quintics and septics. In Section \ref{sub_gen}, considering low degree hypersurfaces containing the projections of suitable Veronese varieties and elliptic normal curves, we introduce techniques to compute the decompositions in the subgeneric cases. For instance, we give an algorithm to compute the two decompositions of a general plane sextic of rank nine. In Section \ref{vsp}, we extend our methods to polynomials admitting infinitely many decompositions using the concept of variety of sums of powers. In Section \ref{Mix} we describe the natural generalization of our main method to general tensors. Finally, in Section \ref{mag_s} we explain how our Magma functions work, and we give some examples on how to use them.

\subsection*{Acknowledgments}
We thank Luca Chiantini and Giorgio Ottaviani for helpful comments. The first named author was partially supported by Proyecto FONDECYT Regular N. 1190777. The second named author is a member of the Gruppo Nazionale per le Strutture Algebriche, Geometriche e le loro Applicazioni of the Istituto Nazionale di Alta Matematica "F. Severi" (GNSAGA-INDAM).

\section{Flattenings and the catalecticant method}\label{Cat}
Let $\underline{n}=(n_1,\dots,n_p)$ and $\underline{d} = (d_1,\dots,d_p)$ be two $p$-uples of positive integers.
Set 
$$d=d_1+\dots+d_p,\ n=n_1+\dots+n_p,\ {\rm and}\
N(\underline{n},\underline{d})=\prod_{i=1}^p\binom{n_i+d_i}{n_i}-1.$$ 

Let $V_1,\dots, V_p$ be $K$-vector spaces of dimensions $n_1+1\leq n_2+1\leq \dots \leq n_p+1$, and consider the product
$$
\mathbb{P}^{\underline{n}} = \mathbb{P}(V_1^{*})\times \dots \times \mathbb{P}(V_p^{*}).
$$
The line bundle 
$$
\mathcal{O}_{\mathbb{P}^{\underline{n}} }(d_1,\dots, d_p)=\mathcal{O}_{\mathbb{P}(V_1^{*})}(d_1)\boxtimes\dots\boxtimes \mathcal{O}_{\mathbb{P}(V_p^{*})}(d_p)
$$
induces an embedding
$$
\begin{array}{cccc}
\sigma\nu_{\underline{d}}^{\underline{n}}:
&\mathbb{P}(V_1^{*})\times \dots \times \mathbb{P}(V_p^{*})& \longrightarrow &
\mathbb{P}(\Sym^{d_1}V_1^{*}\otimes\dots\otimes \Sym^{d_p}V_p^{*})
=\mathbb{P}^{N(\underline{n},\underline{d})-1},\\
      & (\left[v_1\right],\dots,\left[v_p\right]) & \longmapsto & [v_1^{d_1}\otimes\dots\otimes v_p^{d_p}]
\end{array}
$$ 
where $v_i\in V_i$.
We call the image 
$$
\mathcal{SV}_{\underline{d}}^{\underline{n}}= \sigma\nu_{\underline{d}}^{\underline{n}}(\mathbb{P}^{\underline{n}} ) \subset \mathbb{P}^{N(\underline{n},\underline{d})-1}
$$ 
a \textit{Segre-Veronese variety}. When $p = 1$, $\mathcal{V}_{d}^{n}:=\mathcal{SV}_{d}^{n}$ is a Veronese variety. In this case we write $\mathcal{V}_d^n$ for $\mathcal{SV}_{d}^{n}$, and $\nu_{d}^{n}$ for the Veronese embedding.
When $d_1 = \dots = d_p = 1$, $\mathcal{S}^{\underline{n}}:=\mathcal{SV}_{1,\dots,1}^{\underline{n}}$ is a Segre variety. 
In this case we write $\mathcal{S}^{\underline{n}}$ for $\mathcal{SV}_{1,\dots,1}^{\underline{n}}$, and $\sigma^{\underline{n}}$ for the Segre embedding.
Note that 
$$
\sigma\nu_{\underline{d}}^{\underline{n}}=\sigma^{\underline{n}'}\circ \left(\nu_{d_1}^{n_1}\times \dots \times \nu_{d_p}^{n_p}\right),
$$
where $\underline{n}'=(N(n_1,d_1),\dots,N(n_p,d_p))$.

\begin{Remark}\label{pd} If a polynomial $F\in K[x_0,...,x_n]_d$ admits a decomposition in $h$ powers then $F\in\Sec_h(\mathcal{V}_{d}^{n})$, and conversely a general $F\in\Sec_h(\mathcal{V}_{d}^{n})$ can be written as a sum of $h$ powers. If 
$$
F = \lambda_1L^{d}_{1}+...+\lambda_hL^{d}_{h}
$$
is a decomposition then the partial derivatives of order $s$ of $F$ can be decomposed as a linear combination of $L^{d-s}_{1},...,L^{d-s}_{h}$ as well.

These partial derivatives are $\binom{n+s}{n}$ homogeneous polynomials of degree $d-s$ spanning a linear space $H_{\partial F}^{s}\subseteq \mathbb{P}(K[x_0,...,x_n]_{d-s})$. Therefore, the linear space $\left\langle L_1^{d-s},\dots,L_h^{d-s}\right\rangle$ contains $H_{\partial F}^{s}$.
\end{Remark}

\subsection{Flattenings}\label{flat}
Let $V_1,\dots,V_{p}$ be $K$-vector spaces of finite dimension, and consider the tensor product $V_1\otimes ...\otimes V_{p} = (V_{a_1}\otimes ...\otimes V_{a_s})\otimes (V_{b_1}\otimes ...\otimes V_{b_{p-s}})= V_{A}\otimes V_{B}$ with $A\cup B = \{1,...,p\}$, $A=\{a_1,\dots,a_p\}$ and $B = A^c = \{b_1,\dots, b_{p-s}\}$. Then we may interpret a tensor 
$$T \in V_1\otimes ...\otimes V_p = V_{A}\otimes V_{B}$$
as a linear map $\widetilde{T}:V_{A}^{*}\rightarrow V_{A^c}$. Clearly, if the rank of $T$ is at most $r$ then the rank of $\widetilde{T}$ is at most $r$ as well. Indeed, a decomposition of $T$ as a linear combination of $r$ rank one tensors yields a linear subspace of $V_{A^c}$, generated by the corresponding rank one tensors, containing $\widetilde{T}(V_A^{*})\subseteq V_{A^c}$. The matrix associated to the linear map $\widetilde{T}$ is called an \textit{$(A,B)$-flattening} of $T$.    

In the case of mixed tensors we can consider the embedding
$$\Sym^{d_1}V_1\otimes ...\otimes \Sym^{d_p}V_p\hookrightarrow V_A\otimes V_B$$
where $V_A = \Sym^{a_1}V_1\otimes ...\otimes \Sym^{a_p}V_p$,
$V_B=\Sym^{b_1}V_1\otimes ...\otimes\Sym^{b_p}V_p$, with $d_i =
a_i+b_i$ for any $i = 1,...,p$. In particular, if $n = 1$ we may
interpret a tensor $F\in \Sym^{d_1}V_1$ as a degree $d_1$ homogeneous
polynomial on $\mathbb{P}(V_1^*)$. In this case the matrix associated
to the linear map $\widetilde{F}:V_A^*\rightarrow V_B$ is nothing but
the $a_1$-th \textit{catalecticant matrix} of $F$, that is the matrix
whose lines are the coefficient of the partial derivatives of order
$a_1$ of $F$. This identifies the linear space $H_{\partial F}^{s}$ in
Remark \ref{pd} with $\mathbb{P}(\widetilde{F}(V_A^*))\subseteq
\mathbb{P}(V_B)$, where $a_1 = s$, $b_1 = d-a_1 = d-s$.

\begin{Proposition}\cite[Propositions 3.1, 3.2]{MMS18}\label{prop1gen}
Let $F\in k[x_0,...,x_n]_d$ be a polynomial admitting a decomposition $F = \sum_{i=1}^h\lambda_iL_i^d$, $s$ an integer such that ${n+s\choose n}\geq h>{n+s-1\choose n} $, and assume that
\begin{itemize}
\item[i)] the linear space $H_{\partial F}^{s}$ generated by the partial derivatives of order $s$ of $F$ has dimension $h-1$,
\item[ii)] $\dim(H_{\partial F}^{s}\cap \mathcal{V}_{d-s}^n) = 0$,
\item[iii)]  $\deg(H_{\partial F}^{s}\cap \mathcal{V}_{d-s}^n) = h$.
\end{itemize}
Then $F$ is $h$-identifiable and it has rank $h$.  Furthermore, the criterion is effective when ${n+d-s\choose n}>h+n$.  
\end{Proposition}

\begin{Alg}(Catalecticant Algorithm \cite[Section 5.4]{IK99})\label{cat_alg}\\
\textbf{Input:} $F\in K[x_0,\dots,n_n]_d$ admitting a decomposition in $h$ powers.
\begin{itemize}
\item[-] Construct the subspace $H_{\partial F}^{s}\subset\mathbb{P}^{N(n,d-s)}$ with $s = \lceil\frac{d}{2}\rceil$. 
\item[-] If either $\dim(H_{\partial F}^{s})\neq h-1$ or $\dim(H_{\partial F}^{s}\cap \mathcal{V}_{d-s}^n) \neq 0$ or $\deg(H_{\partial F}^{s}\cap \mathcal{V}_{d-s}^n) \neq h$ the algorithm fails. 
\item[-] Otherwise compute the intersection $H_{\partial F}^{s}\cap \mathcal{V}_{d-s}^n = \{L_1^{d-s},\dots,L_{h}^{d-s}\}$.
\item[-] Solve the linear system $F = \sum_{i=1}^h \lambda_iL_i^d$ in the unknowns $\lambda_i\in K$.
\end{itemize}
\end{Alg}

Proposition \ref{prop1gen} can be extended to the mixed case as follows. 

\begin{Proposition}\cite[Propositions 3.1, 3.2]{MMS18}\label{prop2gen}
Let $T\in \Sym^{d_1}V_1\otimes ...\otimes \Sym^{d_p}V_p$ be a tensor
admitting a decomposition $T = \sum_{i=1}^h\lambda_i U_i$. Fix an $(A,B)$-flattening $\widetilde{T}:V_A^*\rightarrow V_B$ of $T$ such that $N(\underline{n},\underline{a})\geq h$, and assume that
\begin{itemize}
\item[i)] the linear space $\mathbb{P}(\widetilde{T}(V_A^*))$ has dimension $h-1$,
\item[ii)] $\dim(\mathbb{P}(\widetilde{T}(V_A^*))\cap \mathcal{SV}_{\underline{b}}^{\underline{n}}) = 0$,
\item[iii)]  $\deg(\mathbb{P}(\widetilde{T}(V_A^*))\cap \mathcal{SV}_{\underline{b}}^{\underline{n}}) = h$.
\end{itemize}
where $\underline{b} = (b_1,...,b_n)$. Then $T$ is $h$-identifiable and it has rank $h$. Furthermore, the criterion is effective when
$N(\underline{n},\underline{b})>h+\dim(\mathcal{SV}_{\underline{b}}^{\underline{n}})$.
\end{Proposition}

\begin{Alg}(Catalecticant Algorithm for Segre-Veronese)\label{cat_alg_sv}\\
\textbf{Input:} $T\in \Sym^{d_1}V_1\otimes ...\otimes \Sym^{d_p}V_p$ admitting a decomposition in $h$ rank one tensors $U_i = v_{1,i}^{d_1}\otimes\dots\otimes v_{p,i}^{d_p}$ for $i = 1,\dots,h$.
\begin{itemize}
\item[-] Fix an $(A,B)$-flattening $\widetilde{T}:V_A^*\rightarrow V_B$ of $T$ such that $N(\underline{n},\underline{a})\geq h$, and consider the subspace $\mathbb{P}(\widetilde{T}(V_A^*))$. 
\item[-] If either $\dim(\mathbb{P}(\widetilde{T}(V_A^*))\neq h-1$ or $\dim(\mathbb{P}(\widetilde{T}(V_A^*))\cap \mathcal{SV}_{\underline{b}}^{\underline{n}}) \neq 0$ or $\deg(\mathbb{P}(\widetilde{T}(V_A^*))\cap \mathcal{SV}_{\underline{b}}^{\underline{n}}) \neq h$ for all $(A,B)$-flattenings the algorithm fails. 
\item[-] Otherwise compute the intersection $\mathbb{P}(\widetilde{T}(V_A^*))\cap \mathcal{SV}_{\underline{b}}^{\underline{n}} = \{U_1^{b},\dots,U_h^{b}\}$ where $U_i^{\underline{b}} = v_{1,i}^{b_1}\otimes\dots\otimes v_{p,i}^{b_p}$.
\item[-] Solve the linear system $F = \sum_{i=1}^h \lambda_i U_i$, with $U_i = v_{1,i}^{d_1}\otimes\dots\otimes v_{p,i}^{d_p}$, in the unknowns $\lambda_i\in K$.
\end{itemize}
\end{Alg}

\section{Generalized catalecticant method}\label{GCM}
In this section we introduce our main method in the symmetric case. As a warm-up we begin by considering first partial derivatives.

\begin{Proposition}\label{gen}
Let $F\in K[x_0,\dots,x_n]_d$ a polynomial admitting a decomposition of the form $F = \sum_{i=1}^h\lambda_iL_i^d$. If
\begin{itemize}
\item[(i)] $H_{\partial F}^1\cap\Sec_{h-n}(\mathcal{V}_{d-1}^n)$ has dimension zero and degree $\binom{h}{n}$, and
\item[(ii)] $H_{\partial F}^1\cap\Sec_{h-n-1}(\mathcal{V}_{d-1}^n)$ contains less than $\binom{h-1}{n}$ points 
\end{itemize} 
then $F$ has rank $h$ and it is $h$-identifiable.
\end{Proposition}
\begin{proof}
Write 
$$F = (\alpha_0^1x_0+\dots+\alpha_n^1x_n)^d+\dots + (\alpha_0^hx_0+\dots+\alpha_n^hx_n)^d$$ 
Then
$$\frac{\partial F}{\partial x_j} = d(\alpha_j^1L_1^{d-1}+\dots + \alpha_j^hL_h^{d-1})$$
where $L_i = \alpha_0^ix_0+\dots+\alpha_n^ix_n$. For $\xi = [\xi_0:\dots:\xi_n]\in\mathbb{P}^n$ we have
$$\xi_0\frac{\partial F}{\partial x_0}+\dots + \xi_n\frac{\partial F}{\partial x_n} = d(L_1(\xi)L_1^{d-1}+\dots +L_h(\xi)L_{h}^{d-1})$$
Any point $D(F)(\xi)\in H_{\partial F}^1$ can be written in the above form. Therefore, if $n$ of the linear forms $L_i$ vanish at a point $\xi\in\mathbb{P}^n$ then $D(F)(\xi) \in \Sec_{h-n}(\mathcal{V}_{d-1}^n)$. This determines $\binom{h}{n}$ points in the intersection $H_{\partial F}^1\cap\Sec_{h-n}(\mathcal{V}_{d-1}^n)$  .

Now, if $F$ is a linear combination of $h-1$ powers of linear forms the same argument will determine $\binom{h-1}{n}$ points in $H_{\partial F}^1\cap\Sec_{h-n-1}(\mathcal{V}_{d-1}^n)$, and this contradicts (ii).

Furthermore, if $F = \sum_{i=1}^h\lambda_iL_i^d = \sum_{j=1}^{h}\mu_j l_j^d$ admits two different decompositions the argument above shows that we would have more than $\binom{h}{n}$ points in $H_{\partial F}^1\cap(\Sec_{h-n}(\mathcal{V}_{d-1}^n))$, and this would contradict (i).
\end{proof}

\begin{Remark}\label{rem_gen}
Since, with the exceptions in Alexander-Hirshowitz's theorem \cite{AH95}, we have that $\dim(\Sec_{h-n}(\mathcal{V}_{d-1}^n)) = (h-n)n+h-n-1$, and whenever the partial derivatives of $F$ are independent $\dim(H^1_{\partial F}) = n$, for condition (i) in Proposition \ref{gen} to hold we must have 
$$h < B_{n,d} := \frac{\binom{d-1+n}{n}+n^2}{n+1}.$$

The catalecticant method, which works at its best for even degree $d = 2k$, produces the bound $h\leq \binom{n+k}{k}$. Note that this binomial coefficients is in general much smaller than $B_{n,d}$.  
\end{Remark}

\begin{Remark}\label{rem_cub}
Take $d = 3$. Since $\codim_{\mathbb{P}^{N(n,2)}}\Sec_{h-n}(\mathcal{V}_2^n) = \frac{n^2+3n-(h-n)(3n-h+3)+2}{2}$ for condition (i) in Proposition \ref{gen} to hold we need to have $h < \frac{4n-\sqrt{8n+1}+3}{2}$. The equality $h = \frac{4n-\sqrt{8n+1}+3}{2}$ is also admissible for $n = 1,3$ since in these cases we have $\codim_{\mathbb{P}^{N(n,2)}}\Sec_{h-n}(\mathcal{V}_2^n) = n$ but also $\deg(\Sec_{h-n}(\mathcal{V}_2^n)) = \binom{h}{n}$. In general $\codim_{\mathbb{P}^{N(n,2)}}\Sec_{h-n}(\mathcal{V}_2^n) = n$ holds if and only if $n$ is a triangular number, that is of the form $\binom{k+1}{2}$. However, if the codimension of $\Sec_{h-n}(\mathcal{V}_2^n)$ is $n$ and $n > 3$ then $\deg(\Sec_{h-n}(\mathcal{V}_2^n)) > \binom{h}{n}$. So condition (i) in Proposition \ref{gen} can not hold. 

When $(n,h) \in \{(1,2),(3,5)\}$ we have that $\deg(\Sec_{h-n}(\mathcal{V}_2^n)) = \binom{h}{n}$. The case $(n,h) = (3,5)$ is known as Sylvester's pentahedral theorem.
\end{Remark}

\begin{Lemma}\label{comb}
Let $\Pi_1,\dots,\Pi_h \subset\mathbb{P}^n$ be $h\geq n+1$ general hyperplanes. Consider the points $\xi_{j_1,\dots,j_{n}} = \Pi_{j_1}\cap\dots\cap\Pi_{j_{n}}$. If $\Pi'\subset\mathbb{P}^n$ is a hyperplane containing $\binom{h-1}{n-1}$ of the $\xi_{j_1,\dots,j_{n}}$ then $\Pi' = \Pi_i$ for some $i = 1,\dots,h$.
\end{Lemma}
\begin{proof}

If $h=n+1$ then $\binom{h-1}{n-1}=n$ and the result is straightforward. If $n=2$ and $h=4$ the result is immediate. Hence, we may assume that $h< {\binom{h-1}{n-1}}$.

By hypothesis $\Pi'$ contains $\binom{h-1}{n-1}$ of the $\xi_{j_1,\dots,j_{n-1}}$, let us denote them by $\xi_1,\dots,\xi_{\binom{h-1}{n-1}}$. In order to cut out $\xi_j$ we must choose $n$ of the $\Pi_j$ and after choosing $h$ of the points at least one of the $\Pi_j$, say $\Pi_1$ has been chosen at least $n$ times. Since $\Pi_1,\dots,\Pi_h$ are general the $n$ points $\xi_{i_1},\dots,\xi_{i_n}\in \Pi'\cap\Pi_1$ determined by this procedure are in liner general position. So $\Pi' = \Pi_1$. 
\end{proof}

Proposition \ref{gen} suggests the following algorithm for computing the linear forms $L_i$ starting from the polynomial $F$. 

\begin{Alg}\label{cat_g}
\textbf{Input:} $F\in K[x_0,\dots,n_n]_d$ admitting a decomposition in $h$ powers.
\begin{itemize}
\item[-] Compute the intersection $H_{\partial F}^1\cap\Sec_{h-n}(\mathcal{V}_{d-1}^n)$. If the hypotheses of Proposition \ref{gen} are not satisfied then the method fails. 
\item[-] Otherwise the points 
$$\xi_1,\dots,\xi_{\binom{h}{n}}\in H_{\partial F}^1\cap\Sec_{h-n}(\mathcal{V}_{d-1}^n)$$
are the points where $n$ of the $h$ linear forms $L_{1},\dots,L_h$ vanish. Note that on each hyperplane $H_i = \{L_i = 0\}$ there are $\binom{n-1}{h-1}$ of the $\xi_i$. 
\item[-] Among all the sets of $\binom{n-1}{h-1}$ of the $\xi_i$ compute those spanning a hyperplane. 
\item[-] By Lemma \ref{comb} these sets are exactly $h$ and the $h$ hyperplanes spanned by them are the zero loci of the linear forms $L_1,\dots,L_h$. 
\item[-] Solve the linear system $F = \sum_{i=1}^h \lambda_iL_i^d$ in the unknowns $\lambda_i\in K$.
\end{itemize}
\end{Alg}

Next, we generalize Proposition \ref{gen} using higher order partial derivatives.

\begin{Lemma}\label{dersp}
Consider a polynomial $F = \sum_{i=1}^h L_i^d\in K[x_0,\dots,x_n]_d$. Then 
$$
D_{\xi}^s(F):=\sum_{s_0\leq\dots \leq s_n}\xi_{s_0,\dots,s_n}\frac{\partial^s F}{\partial x_0^{s_0}\dots\partial x_n^{s_n}} = \frac{d!}{(d-s)!}\sum_{i = 1}^h \left\langle L_i^s,\xi\right\rangle L_j^{d-s}
$$
for all $\xi = (\xi_{s_0,\dots,s_n})\in\mathbb{P}^{N_s}$, where $N_s = \binom{n+s}{s}-1$.
\end{Lemma}
\begin{proof}
Write $L_i = \alpha_{0,i} x_0+\dots +\alpha_{n,i} x_n$. Then 
$$
\frac{\partial^s F}{\partial x_0^{s_0}\dots \partial x_n^{s_n}} = \frac{d!}{(d-s)!}((\alpha_{0,1}^{s_0}\dots \alpha_{n,1}^{s_n})L_1^{d-s}+\dots +(\alpha_{0,h}^{s_0}\dots \alpha_{n,h}^{s_n})L_{h}^{d-s})
$$
Hence 
$$
\begin{array}{ll}
\sum_{s_0\leq\dots \leq s_n}\xi_{s_0,\dots,s_n}\frac{\partial^s F}{\partial x_0^{s_0}\dots\partial x_n^{s_n}} = & \frac{d!}{(d-s)!}\left(\sum_{s_0\leq\dots\leq s_n}\xi_{s_0,\dots,s_n}\sum_{i=1}^h(\alpha_{0,i}^{s_0}\dots \alpha_{n,i}^{s_n})L_i^{d-s}\right)  \\ 
 & \frac{d!}{(d-s)!}\sum_{i=1}^h\sum_{s_0\leq\dots\leq s_n}(\alpha_{0,i}^{s_0}\dots \alpha_{n,i}^{s_n})\xi_{s_0,\dots,s_n}L_{i}^{d-s}
\end{array} 
$$
that is $D_{\xi}^s(F) = \frac{d!}{(d-s)!}\sum_{i=1}^{h}\left\langle L_i^s,\xi\right\rangle L_{i}^{d-s}$.
\end{proof}

\begin{thm}\label{gen2}
Let $F\in K[x_0,\dots,x_n]_d$ be a homogeneous polynomial admitting a decomposition of the form $F = \sum_{i=1}^h L_i^d$. If
\begin{itemize}
\item[(i)] $H_{\partial F}^s\cap\Sec_{h-N_s}(\mathcal{V}_{d-s}^n)$ has dimension zero and degree $\binom{h}{N_s}$, and
\item[(ii)] $H_{\partial F}^s\cap\Sec_{h-N_s-1}(\mathcal{V}_{d-s}^n)$ contains less than $\binom{h-1}{N_s}$ points;
\end{itemize} 
then $F$ has rank $h$ and it is $h$-identifiable. Furthermore, if $h\leq N_s+2$ the identifiability criterion is effective. 
\end{thm}
\begin{proof}
Write 
$$F = (\alpha_0^1x_0+\dots+\alpha_n^1x_n)^d+\dots + (\alpha_0^hx_0+\dots+\alpha_n^hx_n)^d.$$ 
By Lemma \ref{dersp}, if $N_s$ of the linear forms $L_i$ vanish at a point $\xi\in\mathbb{P}^{N_s}$ then $D^s_{\xi}(F) \in \Sec_{h-N_s}(\mathcal{V}_{d-s}^n)$. This determines $\binom{h}{N_s}$ points in the intersection $H_{\partial F}^s\cap\Sec_{h-N_s}(\mathcal{V}_{d-s}^n)$.

Now, if $F$ is a linear combination of $h-1$ powers of linear forms the same argument will determine $\binom{h-1}{N_s}$ points in $H_{\partial F}^s\cap\Sec_{h-N_s-1}(\mathcal{V}_{d-s}^n)$, and this contradicts (ii).

Furthermore, if $F = \sum_{i=1}^h\lambda_iL_i^d = \sum_{j=1}^{h}\mu_j l_j^d$ admits two different decompositions the argument above shows that we would have more than $\binom{h}{N_s}$ points in $H_{\partial F}^s\cap\Sec_{h-N_s}(\mathcal{V}_{d-s}^n)$, and this would contradict (i).

Finally, we prove the effectiveness of the criterion for $h\leq N_s+2$. It is enough to consider the case $h = N_s+2$. Take $F\in \Sec_{h}(\mathcal{V}_d^n)$ general, let $H_L$ be the $(h-1)$-plane spanned by the powers $L_i^{d-s}$, and assume that $H_{\partial F}^s$ intersects $\Sec_{h-N_s}(\mathcal{V}_{d-s}^n)$ in an additional point $G$. Then we may write $G = \sum_{j=1}^{h-N_s}\alpha_j l_j^{d-s}$. Set $H_l = \left\langle H_{\partial F}^s,l_1^{d-s},\dots,l_{h-N_s}^{d-s}\right\rangle$. Note that $\dim(H_l)\leq N_s+(h-N_s-1) = h-1$. We may write the $N_s+1$ partial derivatives of order $s$ of $F$ as linear combinations of the $L_i^{d-s}$. Moreover, for the polynomial $G$ we have $G = \sum_{i=1}^{h}\lambda_iL_i^{d-s} = \sum_{j=1}^{h-N_s}\alpha_j l_j^{d-s}$. Hence, we have $N_s+1$ linear equations in the $L_1^{d-s},\dots, L_h^{d-s}$, and since $h = N_s+2$ we may write $L_1^{d-s},\dots, L_h^{d-s}$ as linear combinations of the partial derivatives of order $s$ of $F$ and of $l_1^{d-s},\dots,l_{h-N_s}^{d-s}$. Then $H_L = H_l$ intersects $\mathcal{V}_{d-s}^{n}$ at least in $\{L_1^{d-s},\dots, L_h^{d-s},l_1^{d-s},\dots,l_{h-N_s}^{d-s}\}$. On the other hand, $H_L$ is generated by the $h$ general points $L_1^{d-s},\dots, L_h^{d-s}\in \mathcal{V}_{d-s}^{n}$ and this contradicts the Trisecant lemma \cite[Proposition 2.6]{CC02}.
\end{proof}

\begin{Remark}\label{eff_cub}
As observed in Remark \ref{rem_cub}, Theorem \ref{gen2} gives a criterion for identifiability of cubics for $h < \frac{4n-\sqrt{8n+1}+3}{2}$. Furthermore, by the last part of Theorem \ref{gen} we have that such criterion is effective for $h\leq n+2$. 
\end{Remark}

\section{Uniqueness and finiteness of the decompositions}\label{sec:Hilb}
Let $h(n,d)$ be the minimum integer such that a general $F\in
k[x_0,...,x_n]_d$ admits a decomposition in sum of powers. The number
$h(n,d)$ has been determined in \cite{AH95} and $h(n,d)$-identifiability very seldom holds.
Indeed, by \cite[Theorem 1]{GM19} a general polynomial $F\in k[x_{0},...,x_{n}]_{d}$ is $h(n,d)$-identifiable only in the following cases:
\begin{itemize}
\item[-] $n = 1$, $d = 2m+1$, $h(n,d)=m$ \cite{Sy04};
\item[-] $n = d = 3$, $h(3,3) = 5$ \cite{Sy04};
\item[-] $n = 2$, $d = 5$, $h(2,5) = 7$ \cite{Hi88}.
\end{itemize}
The case $n = 1$ is covered by the catalecticant method, while the second case can be achieved using Proposition \ref{gen}. The next result deals with the third case. 

\begin{Proposition}\label{Hilb}
Let $F\in K[x_0,\dots,x_n]_d$ admitting a decomposition of the form admitting a decomposition of the form $F = \sum_{i=1}^h\lambda_iL_i^d$. Assume that for some $s$ the linear space $H^s_{\partial F}\subset\mathbb{P}(K[x_0,x_1,x_2]_{d-s})$ has dimension $\binom{n+s}{s}-1 = h-2$ and does not intersect $\mathcal{V}_{d-s}^n$. Set $\overline{V}_{d-s}^n := \pi_{H^s_{\partial F}}(\mathcal{V}_n^d)$, where $\pi_{H^s_{\partial F}}:\mathbb{P}^{N(n,d-s)}\dasharrow\mathbb{P}^{M}$ is the projection from $H^s_{\partial F}$. If $\overline{V}_{d-s}^n$ has a unique point of multiplicity $h$ then $F$ is $h$-identifiable. 
\end{Proposition} 
\begin{proof}
Since $H^s_{\partial F}\cap \mathcal{V}_{d-s}^n = \emptyset$ the projection $\pi_{H^s_{\partial F}}$ restricts to a morphism on $\mathcal{V}_{d-s}^n$. Write $F = \sum_{i = 1}^{h}\lambda_iL_i^d$, and set $H_{L} = \left\langle L_1^{d-s},\dots, L_h^{d-s}\right\rangle$. Then $H^s_{\partial F}$ is a hyperplane in $H_{L}$, and $\pi_{H^s_{\partial F}}(H_L)\in \overline{V}_{d-s}^n$ is a singular point of multiplicity $h$ for $\overline{V}_{d-s}^n$. Assume that $F = \sum_{i = 1}^{h}\lambda_iL_i^d = \sum_{i = 1}^{h}\mu_il_i^d$ admits two different decompositions and consider the associated $(h-1)$-planes $H_L,H_l$ in $\mathbb{P}(K[x_0,x_1,x_2]_{d-s})$. Then $\pi_{H^s_{\partial F}}(H_L),\pi_{H^s_{\partial F}}(H_l)\in \overline{V}_{d-s}^n$ are two points of multiplicity $h$ for $\overline{V}_{d-s}^n$, a contradiction. 
\end{proof}

Note that Proposition \ref{Hilb} provides the following algorithm to compute the decomposition.

\begin{Alg}\label{Alg_Hilb}
\textbf{Input:} $F\in K[x_0,\dots,n_n]_d$ admitting a decomposition in $h$ powers.
\begin{itemize}
\item[-] If either $\dim(H^s_{\partial F}) \neq h-2$ or $H^s_{\partial F}\cap \mathcal{V}_{d-s}^n\neq\emptyset$ for all $s$ the algorithm fails. Otherwise, consider the linear space $H^s_{\partial F}$ such that $\dim(H^s_{\partial F}) = h-2$ and $H^s_{\partial F}\cap \mathcal{V}_{d-s}^n = \emptyset$.
\item[-] Compute the image of the projection $\overline{V}_{d-s}^n := \pi_{H^s_{\partial F}}(\mathcal{V}_{d-s}^n)$.
\item[-] Compute the reduced subscheme $S_h\subset\overline{V}_{d-s}^n$ consisting of the points of multiplicity $h$ of $\overline{V}_{d-s}^n$. 
\item[-] If $S_h$ consists of more than one point the algorithm fails. If $S_h = \{p\}$ consists of a single point compute the linear span $H_p = \left\langle H^s_{\partial F},p\right\rangle \subset \mathbb{P}^{N(n,d-s)}$.
\item[-] Compute the intersection $H_{p}\cap \mathcal{V}_{d-s}^n = \{L_1^{d-s},\dots,L_h^{d-s}\}$. 
\item[-] Solve the linear system $F = \sum_{i=1}^h \lambda_iL_i^d$ in the unknowns $\lambda_i\in K$.
\end{itemize}
\end{Alg}

\begin{Remark}\label{rem_Hilb}
In particular, when $n = 2,d = 5,h = 7, s = 2$ Proposition \ref{Hilb} provides an algorithm to compute the decomposition of a plane quintic in seven powers. 
\end{Remark}

\subsection*{Decompositions over the rationals}
Let $F\in \mathbb{Q}[x_0,\dots,x_n]_d$ be a homogeneous polynomial admitting a decomposition, as sum of powers, defined over $\mathbb{Q}$. In this case the methods in Section \ref{GCM} can by extended using star configurations to the cases where the intersection $H_{\partial F}^s\cap \Sec_{h-N_s}(\mathcal{V}_{d-s}^n)$ has positive dimension. 

\begin{Definition}\label{star}
Let $\mathcal{H} = \{H_1,\dots, H_m\}$ be a collection of $m$ distinct hyperplanes in $\mathbb{P}^r$. Assume that the intersection of any $t$ of these hyperplanes is either empty or has codimension $t$. For any $1\leq c\leq \min(m,n)$ the codimension $c$ star configuration associated to $\mathcal{H}$ is the union 
$$\mathcal{S}_c(\mathcal{H},\mathbb{P}^r):= \bigcup_{1\leq i_1<\dots < i_c\leq s}H_{i_1}\cap\dots \cap H_{i_c}$$ 
of the codimension $c$ linear subspaces defined by all the intersections of $c$ of the hyperplanes in $\mathcal{H}$.
\end{Definition}

We refer to \cite{AV11}, \cite{GHM13}, \cite{CGV14}, \cite{CGV15} for details on star configurations.

\begin{Proposition}\label{star_dec}
Let $F\in K[x_0,\dots,x_n]_d$ be a homogeneous polynomial admitting a decomposition of the form $F = \sum_{i=1}^{h}\lambda_iL_i^d$. Fix homogeneous coordinates $\xi = (\xi_{s_0,\dots,s_n})$ on $H_{\partial F}^s$, and set $H_i = \{\left\langle L_i^s,\xi\right\rangle =0\}\subset H_{\partial F}^s$ for $i = 1,\dots,h$. For any $c\leq\dim(H_{\partial F}^s)$ the collection of hyperplanes $\mathcal{H} = \{H_1,\dots,H_{\binom{h}{c}}\}$ defines a codimension $c$ star configuration $\mathcal{S}_c(\mathcal{H}, H_{\partial F}^s)$ contained in $H_{\partial F}^s\cap \Sec_{h-c}(\mathcal{V}_{d-s}^n)$.
\end{Proposition} 
\begin{proof}
Since the linear forms $L_i$ are linearly independent the polynomials $L_i^{s}$ are linearly independent in $K[x_0,\dots,x_n]_{s}$. So $\mathcal{S}_c(\mathcal{H}, H_{\partial F}^s)$ is a star configuration of codimension $c$ in $H_{\partial F}^s$.

Furthermore, by Lemma \ref{dersp} the codimension $c$ linear subspace in $\mathcal{S}_c(\mathcal{H}, H_{\partial F}^s)$ are contained in $\Sec_{h-c}(\mathcal{V}_{d-s}^n)$, and hence $\mathcal{S}_c(\mathcal{H}, H_{\partial F}^s)\subset H_{\partial F}^s\cap \Sec_{h-c}(\mathcal{V}_{d-s}^n)$.
\end{proof}

Proposition \ref{star_dec} is particularly interesting when $F\in \mathbb{Q}[x_0,\dots,x_n]_d$ has a decomposition defined over $\mathbb{Q}$, and $c = \dim(H_{\partial F}^s)$. In this case $\mathcal{S}_c(\mathcal{H}, H_{\partial F}^s)$ is a star configuration of points in $H_{\partial F}^s\cap \Sec_{h-c}(\mathcal{V}_{d-s}^n)$, and since the linear forms $L_i$ are defined over $\mathbb{Q}$ we have that the points of $\mathcal{S}_c(\mathcal{H}, H_{\partial F}^s)$ are defined over $\mathbb{Q}$ as well. 

\begin{Alg}\label{over_Q}
\textbf{Input:} $F\in \mathbb{Q}[x_0,\dots,n_n]_d$ admitting a decomposition in $h$ powers defined over $\mathbb{Q}$.
\begin{itemize}
\item[-] Compute the intersection $H_{\partial F}^s\cap \Sec_{h-N_s}(\mathcal{V}_{d-s}^n)$.
\item[-] Compute the rational points of $H_{\partial F}^s\cap \Sec_{h-N_s}(\mathcal{V}_{d-s}^n)$.
\item[-] If $H_{\partial F}^s\cap \Sec_{h-N_s}(\mathcal{V}_{d-s}^n)$ has infinitely many rational points the algorithm fails.
\item[-] If the set of rational points $\{Q_1,\dots,Q_k\}$ of $H_{\partial F}^s\cap \Sec_{h-N_s}(\mathcal{V}_{d-s}^n)$ is finite solve the linear system $F = \sum_{i=1}^h \lambda_iQ_i^d$ for all subsets of cardinality $h$ of $\{Q_1,\dots,Q_k\}$.
\end{itemize}
\end{Alg}

Algorithm \ref{over_Q} works particularly well for plane quintics. 

\begin{Proposition}\label{Q_Quint}
Let $F\in K[x_0,x_1,x_2]_5$ be a homogeneous polynomial. Assume that $H_{\partial F}^{1}$ intersects $\Sec_{5}(\mathcal{V}_{4}^2)$ along a smooth curve. Then $H_{\partial F}^{1}\cap \Sec_{5}(\mathcal{V}_{4}^2)$ has finitely many rational points.  
\end{Proposition}
\begin{proof}
The secant variety $\Sec_{5}(\mathcal{V}_{4}^2)$ is the hypersurface of degree six in $\mathbb{P}^{14}$ cut out by the catalecticant matrix of the second partial derivatives of a polynomial of degree four in three variables \cite[Section 1]{LO13}, and $H_{\partial F}^{1}\cong \mathbb{P}^2$. Hence $C = H_{\partial F}^{1}\cap \Sec_{5}(\mathcal{V}_{4}^2)$ is a smooth plane sextic and by Faltings theorem \cite{Fa83} $C$ has finitely many rational points.
\end{proof}

\begin{Proposition}\label{dix}
Let $F\in K[x_0,\dots,x_n]_d$ be e homogeneous polynomial admitting a decomposition of the form $F = \sum_{i=1}^h\lambda_iL_i^d$. Assume that for some integer $s\geq 0$ we have that $H_{\partial F}^s$ has dimension $h-3$. Set $\overline{H}_L = \overline{\pi_{H_{\partial F}^s}(H_L)}\subset\mathbb{P}^M$, where $H_L = \left\langle L_1^{d-s},\dots,L_{h}^{d-s}\right\rangle$. Then $\overline{H}_L\subset\mathbb{P}^M$ is contained in all the hypersurfaces of degree $c < h$ in the ideal of $\overline{V}_{d-s}^n =\pi_{H_{\partial F}^s}(\mathcal{V}_{d-s}^n)$.
\end{Proposition}
\begin{proof}
Note that $\overline{H}_L\subset\mathbb{P}^M$ is a line intersecting $\overline{V}_{d-s}^n$ in at least $h$ points. Hence, if $Z_c$ is a hypersurface of degree $c< h$ containing $\overline{V}_{d-s}^n$ B\'ezout's theorem yields that $\overline{H}_L\subset Z_c$.
\end{proof}

\begin{Alg}\label{alg_dix}
\textbf{Input:} $F\in K[x_0,\dots,x_n]_d$ admitting a decompositions in $h$ powers. Assume that the hypothesis of Proposition \ref{dix} are satisfied. 
\begin{itemize}
\item[-] Compute $\overline{V}_{d-s}^n$ and the scheme $Z$ cut out by the hypersurfaces of degree at most $h-1$ containing $\overline{V}_{d-s}^n$.
\item[-] If $Z$ does not have components of dimension one the algorithm fails. Otherwise do the following for all components $\overline{R}$ of dimension one of $Z$: 
\begin{itemize}
\item[(i)] let $R$ be the inverse image of $\overline{R}$ via $\pi_{H_{\partial F}^s}:\mathbb{P}^{N(n,d-s)}\dasharrow\mathbb{P}^M$;
\item[(i)] if $R$ intersects $\mathcal{V}^n_{d-s}$ in $h$ points $\{l_1^{d-s},\dots,l_h^{d-s}\}$ solve the linear system $F = \sum_{i=1}^h \lambda_il_i^d$, otherwise go the the next component.
\end{itemize}
\item[-] For any $R$ such that $R\cap \mathcal{V}^n_{d-s}$ consists of $h$ points and $F = \sum_{i=1}^h \lambda_il_i^d$ is compatible we get a decomposition of $F$ in $h$ powers. 
\end{itemize}
\end{Alg}

\begin{Remark}\label{Rem_dix}
The generic rank of a general polynomial of degree seven in three variables is $h = 12$, and there are exactly five decompositions \cite{Di07}. We successfully applied Algorithm \ref{alg_dix} to this case.  
\end{Remark}

\section{Subgeneric cases}\label{sub_gen}
As opposed to the generic case when $h < h(n,d)$ is smaller than the generic rank we have that a general polynomial $F\in K[x_0,\dots,x_n]_d$ of rank $h$ is identifiable with the following three exceptions:
\begin{itemize}
\item[-] $n = 2, d = 6, h = 9$;
\item[-] $n = 3, d = 4, h = 8$;
\item[-] $n = 5, d = 3, h = 9$;
\end{itemize}
and in these three cases there are exactly two decompositions \cite[Theorem 1.1]{COV18}.

\begin{Proposition}\label{hypers}
Let $F\in K[x_0,\dots,x_n]_d$ be a homogeneous polynomial admitting a decomposition of the form $F = \sum_{i=1}^{h}\lambda_iL_i^d$. Assume that $H_{\partial F}^s\cap \mathcal{V}_{d-s}^n = \emptyset$ and consider $\pi_{H_{\partial F}^s}:\mathbb{P}^{N(n,d-s)}\dasharrow\mathbb{P}^M$ the projection with center $H_{\partial F}^s$. Set $H_L = \left\langle L_1^{d-s},\dots,L_h^{d-s}\right\rangle$. Assume that the $L_i^{d-s}$ lie on a subvariety $W\subseteq \mathcal{V}_{d-s}^n$. Set $\overline{H}_L = \overline{\pi_{H_{\partial F}^s}(H_L)}$ and $\overline{W} =\pi_{H_{\partial F}^s}(W)$.

Assume that there are integers $a,b$ such that through $h-a+1$ general points of $\overline{H}_L$ there exists a curve of degree $b$ contained in $\overline{H}_L$. Then $\overline{H}_L$ is contained in all the hypersurfaces of degree $c < \frac{h-a+1}{b}$ in the ideal of $\overline{W}$. 
\end{Proposition}
\begin{proof}
The projection $\overline{H}_L$ of $H_L$ is a linear subspace of dimension $h-\binom{n+s}{s}-1$ intersecting $\overline{V}_{d-s}^n$ in $h$ points $\{x_1,\dots,x_h\}$.

Let $C_b\subset \overline{H}_L$ be a curve of degree $b$ through $x_{1},\dots,x_{h-a},y$, where $y\in \overline{H}_L$ is general, and $Z_{c}\subset\mathbb{P}^M$ a hypersurface of degree $c$ containing $\overline{V}_{d-s}^n$. Note that $Z_c$ intersects $C_b$ in at least $h-a+1$ points. Since $bc < h-a+1$ B\'ezout's theorem yields that $C_b\subset Z_c$. Finally, since $y\in \overline{H}_L$ is general we conclude that $\overline{H}_L\subset Z_c$ as well.  
\end{proof}

\begin{Alg}\label{hyp}
\textbf{Input:} $F\in K[x_0,\dots,n_n]_d$ admitting a decompositions in $h$ powers. Assume that the hypotheses of Proposition \ref{hypers} are satisfied. 
\begin{itemize}
\item[-] Compute the projection $\overline{W}\subset\mathbb{P}^M$ of $W$.
\item[-] Compute the subscheme $Z\subset\mathbb{P}^M$ cut out by the hypersurface of degree $c < \frac{h-a+1}{b}$ in the ideal of $\overline{W}$.
\item[-] If $Z$ has no linear irreducible component of dimension $h-\binom{n+s}{s}-1$ the algorithm fails.
\item[-] Otherwise, do the following for all the linear irreducible components $\overline{R}$ of dimension $h-\binom{n+s}{s}-1$ of $Z$: 
\begin{itemize}
\item[(i)] let $R$ be the inverse image of $\overline{R}$ via $\pi_{H_{\partial F}^s}:\mathbb{P}^{N(n,d-s)}\dasharrow\mathbb{P}^M$;
\item[(i)] if $R$ intersects $\mathcal{V}^n_{d-s}$ in $h$ points $\{l_1^{d-s},\dots,l_h^{d-s}\}$ solve the linear system $F = \sum_{i=1}^h \lambda_il_i^d$, otherwise go the the next component.
\end{itemize}
\item[-] For any $\overline{R}$ such that $R\cap \mathcal{V}^n_{d-s}$ consists of $h$ points and $F = \sum_{i=1}^h \lambda_il_i^d$ is compatible we get a decomposition of $F$ in $h$ powers. 
\end{itemize}
\end{Alg}

\begin{Remark}
In the exceptional cases 
$$(d,n,h)\in\{(6,2,9),(4,3,8),(3,5,9)\}$$
there are exactly two decomposition for the general polynomial. In each of these cases the two decompositions are contained in an elliptic normal curve $C_{d,n,h}$ \cite[Theorem 1.2]{COV18}. Proposition \ref{hypers} can be applied with the following values:
\begin{itemize}
\item[-] for $(d,n,h) = (6,2,9)$ we consider $C_{6,2,9} = H_{\partial F}^{3}\cap \mathcal{V}_{3}^3\subset\mathbb{P}^9$. Take $W = \nu_4^2((\nu_3^2)^{-1}(C_{6,2,9}))\subset\mathbb{P}^{14}$ and $\overline{W} = \pi_{H_{\partial F}^2}(W)\subset\mathbb{P}^8$. 

Then $\overline{H}_L$ is a $2$-plane intersecting $\overline{W}$ in $h = 9$ points. Taking, in Proposition \ref{hypers}, $a=1$, $b = 3$, that is the plane cubics through eight of the nine points and a general point of $\overline{H}_L$, we get that $\overline{H}_L$ is contained in all the hypersurfaces of degree $c\in\{1,2\}$ of $\mathbb{P}^8$ containing $\overline{W}$. 
\item[-] for $(d,n,h) = (3,5,9)$ take $s = 1$ and as before $W\subset \mathcal{V}_{2}^5\subset\mathbb{P}^{20}$ the elliptic curve containing the decomposition of $F$ and $\overline{W} =\pi_{H_{\partial F}^1}(W)\subset\mathbb{P}^{14}$. Again taking $a=1$, $b = 3$ in Proposition \ref{hypers} we get that $\overline{H}_L$ is contained in all the hypersurfaces of degree $c\in\{1,2\}$ of $\mathbb{P}^{14}$ containing $\overline{W}$.
\end{itemize} 
\end{Remark}

For the case $(d,n,h) = (4,3,8)$ a little variation is needed. 

\begin{Proposition}\label{(4,3,8)}
Let $F\in K[x_0,x_1,x_2,x_3]_4$ be a homogeneous polynomial admitting a decomposition of the form $F = \sum_{i=1}^8\lambda_iL_i^4$. Set $H_L = \left\langle L_1^{3},\dots,L_8^{3}\right\rangle$, $\overline{H}_L = \overline{\pi_{H_{\partial F}^1}(H_L)}\cong\mathbb{P}^3$, and let $W\subset\mathbb{P}^{19}$ be the elliptic curve containing the decomposition of $F$. Then the elliptic normal curve $C\subset \overline{H}_L$ through $\pi_{H_{\partial F}^1}(\{L_1^{3},\dots,L_8^{3}\})$ is contained in all the hypersurfaces of degree $c\in\{1,2\}$ in the ideal of $\overline{W}$.
\end{Proposition}
\begin{proof}
Let $Z_c$ be a hypersurface of degree $c$ containing $\overline{W}$. Then $Z_c$ intersects $C$ in at least eight points. So if $c\leq 1$ B\'ezout's theorem yields that $C\subset Z_c$. Write $C = Q_1\cap Q_2$ as the intersection of two quadrics. A quadric $Q$ passing through $\pi_{H_{\partial F}^1}(\{L_1^{3},\dots,L_8^{3}\})$ is in the pencil generated by $Q_1,Q_2$, and hence it must contain $C$. 
\end{proof} 

\begin{Remark}
Thanks to Algorithm \ref{hyp} we managed to compute the two decompositions of a general rank nine plane sextic. In this case a geometric way to find the second decomposition, once one of them is known, has been described in terms of liaison in \cite[Section 3]{CO21}.
\end{Remark}

\section{Infinitely many decompositions}\label{vsp}
In this section, plugging-in the concept of variety of sum of powers, we give algorithm to compute a decomposition of a polynomial admitting infinitely many. 

\begin{Definition}\label{vspord}
Let $F\in k[x_0,...,x_n]_{d}$ be a general homogeneous polynomial of degree $d$. Let $h$ be a positive integer and $\Hilb_h(\mathbb{P}^{n*})$ the Hilbert scheme of sets of $h$ points in $\mathbb{P}^{n*}$. We define 
$$\VSP(F,h)^{o} := \{\{L_{1},...,L_{h}\}\in\Hilb_{h}(\mathbb{P}^{n*})\: | \: F\in \langle L_{1}^d,...,L_{h}^d\rangle\subseteq\mathbb{P}^{N(n,d)}\}\subseteq \Hilb_{h}(\mathbb{P}^{n*})\},$$
and $\VSP(F,h) := \overline{\VSP(F,h)^{o}}$ by taking the closure of $\VSP(F,h)^{o}$ in $\Hilb_{h}(\mathbb{P}^{n*})$. 
\end{Definition}

Assume that the general polynomial $F\in\mathbb{P}^{N(n,d)}$ is contained in a $(h-1)$-linear space $h$-secant to $\mathcal{V}_{d}^{n}$. Then, by \cite[Proposition 3.2]{Do04} the variety $\VSP(F,h)$ has dimension 
$$\dim(\VSP(F,h)) = h(n+1)-N(n,d)-1.$$
Furthermore, if $n = 1,2$ then for $F$ varying in an open Zariski subset of $\mathbb{P}^{N(n,d)}$ the variety $\VSP(F,h)$ is smooth and irreducible.

In order to apply these objects to the study of decompositions, we need to construct similar varieties parametrizing decomposition of homogeneous polynomials as sums of powers of ordered linear forms. Let us consider the incidence variety 
$$
\mathcal{J}:=\{((l_1,\dots,l_s),\{L_1,...,L_h\})\: | \: l_i\in\{L_1,...,L_h\}\in \VSP(F,h)^{o} \text{ for } i = 1,\dots,s\}\subseteq (\mathbb{P}^{n*})^{s}\times \VSP(F,h)^{o}
$$
and define $\VSP_{s}(F,h)$ as the closure $\overline{\mathcal{J}}$ of $\mathcal{J}$ in $(\mathbb{P}^{n*})^{s}\times \VSP(F,h)$.

Let $V$ be a complex vector space of dimension $n+1$, choose coordinates $x_0,\dots,x_n$ on $V$ and the dual coordinates $\xi_0,\dots,\xi_n$ on $V^{*}$. Let $F\in K[x_0,\dots,x_n]_{d}$ be a homogeneous polynomial of even degree $d = 2m$, and consider the basis of $K[x_0,\dots,x_n]_{m}$ given by
\stepcounter{thm}
\begin{equation}\label{basis}
\mathcal{B}= \left\lbrace\binom{m}{m_0,\dots ,m_n}\prod_{t=0}^nx_t^{m_t},\: m_0+\dots + m_n = m\right\rbrace
\end{equation} 
where $\binom{m}{m_0,\dots ,m_n} = \frac{m!}{m_0!\dots m_n!}$. The $m$-th catalecticant matrix $Cat_m(F)$ of $F$ is the $(m+1)\times (m+1)$ symmetric matrix whose rows are the order $m$ partial derivatives of $F$ written in the basis $\mathcal{B}$ (\ref{basis}) in lexicographic order. The matrix $Cat_m(F)$ induces a symmetric bilinear form 
$$\Omega_F:K[\xi_0,\dots,\xi_n]_m\times K[\xi_0,\dots,\xi_n]_m\rightarrow K.$$ 
For our purposes the following result will be fundamental.

\begin{Lemma}\cite[Proposition 3.8]{Do04}\label{key}
Let $F\in K[x_0,\dots,x_n]_{d}$ be a homogeneous polynomial of even degree $d = 2m$ and assume that $F$ can be decomposed as 
$$F = L_1^{2m}+\dots +L_{h}^{2m}$$
and the powers $L_i^m$ are linearly independent in $K[x_0,\dots,x_n]_m$. Then $\Omega_F(L_i^m,L_j^m)=0$ for any $i,j = 1,\dots h$ with $i\neq j$. 
\end{Lemma}

Define $X_{s}$ as the subvariety of $(\mathbb{P}^{n*})^{s}$ cut out by the relations $\Omega_F(p_i,p_j)=0$ for any $i,j = 1,\dots s$ with $i\neq j$, where we denote by $p_i$ a point in the $i$-th factor of $(\mathbb{P}^{n*})^{s}$.

\begin{Proposition}\label{inf_dec}
Assume that 
$$\dim(\VSP(F,h))-\dim(\VSP(F,h-s))\geq ns-\binom{s}{2}.$$
Then for a general $(L_1^{m},\dots,L_s^{m})\in X_s$ there exists a decomposition of $F$ in $h$ linear forms of the form $\{L_1,\dots,L_s,l_{s+1},\dots,l_h\}$.
\end{Proposition}
\begin{proof}
By Lemma \ref{key} the image of the projection onto the first factor
$$\pi: \VSP_{s}(F,h)\rightarrow (\mathbb{P}^{n*})^{s}$$
is contained in $X_s$. Fix $(L_1^{m},\dots,L_s^{m})\in X_s$ general. The fiber of $\pi$ over $(L_1^{m},\dots,L_s^{m})$ is a finite covering of $\VSP(F,h-s)$. Hence, under our hypotheses $\pi: \VSP_{s}(F,h)\rightarrow X_s$ is dominant.   
\end{proof}

\begin{Remark}\label{rem_VSP}
Proposition \ref{inf_dec} says that, under certain hypothesis on the dimensions of the relevant varieties of sums of powers, in order to construct a decomposition of $F$ in $h$ powers we can simply choose $s$ linear forms $L_1,\dots,L_s\in \mathbb{P}^{n*}$ such that $(L_1^{m},\dots,L_s^m)\in X_s$ and then compute a decomposition of $F-L_1^{d}-\dots-L_s^d$ in $h-s$ powers. For instance, Proposition \ref{inf_dec} can be successfully applied in the following cases:
\begin{itemize}
\item[-] $(d,n,h) = (4,2,6)$, with $s = 2$. In this case $X_2\subset(\mathbb{P}^{2*})^{2}$ is a $3$-fold and $\dim(\VSP(F,6)) = 3$. Since a general polynomial $G\in\Sec_4(\mathcal{V}_4^2)$ admits a unique decomposition the map 
$$\pi: \VSP_{2}(F,6)\rightarrow X_2$$
is dominant and finite. Hence we may choose a general point $(L_1^2,L_2^2)\in X_2$ and then reconstruct a decomposition in four powers of $G := F-L_1^4-L_2^4$ using Algorithm \ref{cat_alg}. 
\item[-] $(d,n,h) = (6,2,10)$, with $s = 1$. Since a general polynomial $G\in\Sec_9(\mathcal{V}_6^2)$ admits two decompositions in nine powers and $\VSP(F,10)$ is a surface the map 
$$\pi: \VSP_{1}(F,10)\rightarrow X_1\cong \mathbb{P}^{2*}$$
is dominant and finite. So we may choose a general linear form $L_1\in \mathbb{P}^{2*}$ and compute a decomposition in nine powers of $G := F-L_1^6$ with Algorithm \ref{hyp}.
\item[-] $(d,n,h) = (4,3,10)$, with $s = 2$. Since a general polynomial $G\in\Sec_8(\mathcal{V}_4^3)$ admits two decompositions in nine powers and $\VSP(F,10)$ has dimension five the map 
$$\pi: \VSP_{2}(F,10)\rightarrow X_2$$
is dominant and finite. Again, we may choose a general point $(L_1^2,L_2^2)\in X_2$ and then reconstruct a decomposition in eight powers of $G := F-L_1^4-L_2^2$.
\end{itemize}
\end{Remark}

\subsection{Lifting decompositions from derivatives}

In this section we give conditions ensuring that a simultaneous decomposition of the derivatives of a polynomial lifts to a decomposition of the polynomial itself.

\begin{Lemma}\label{PEul}
Let $F\in K[x_{0},...,x_{n}]_{d}$ be a homogeneous polynomial. Assume that its partial derivatives admit a decomposition
$$F_{x_{0}} = \sum_{i=1}^{h}\alpha_{i}^{0}L_{i}^{d-1},...,F_{x_{n}} = \sum_{i=1}^{h}\alpha_{i}^{n}L_{i}^{d-1}$$
in $h$ linear forms $L_{i} = A^{0}_{i}x_{0}+...+A^{n}_{i}x_{n}$ such that $L_{1}^{d-2},...,L_{h}^{d-2}$ are independent in $K[x_{0},...,x_{n}]_{d-2}$. Then there are the following relations between the coefficients
$$\alpha_{i}^{t}A_{i}^{s} = \alpha_{i}^{s}A_{i}^{t}, \quad t,s = 0,...,n; \quad i=1,...,h.$$
These relations force the decomposition of the partial derivatives to be of the following form
$$F_{x_{0}} = \sum_{i=1}^{h}\alpha_{i}^{0}\lambda_{i}^{d-1}(\alpha_{i}^{0}x_{0}+...+\alpha_{i}^{n}x_{n})^{d-1},...,F_{x_{n}} = \sum_{i=1}^{h}\alpha_{i}^{n}\lambda_{i}^{d-1}(\alpha_{i}^{0}x_{0}+...+\alpha_{i}^{n}x_{n})^{d-1},$$
where $\lambda_{i} = \frac{A_{i}^{0}}{\alpha_{i}^{0}} = ... = \frac{A_{i}^{n}}{\alpha_{i}^{n}}$. Furthermore the decomposition lifts to a decomposition of the polynomial
$$F = \sum_{i=1}^{h}\frac{1}{\lambda_{i}}L_{i}^{d}.$$
\end{Lemma}
\begin{proof}
We have $F_{x_{t}x_{s}} = F_{x_{s}x_{t}}$ for any $t,s = 0,..,n$. Since $L_{1}^{d-2},...,L_{h}^{d-2}$ are independent these equalities forces $\alpha_{i}^{t}A_{i}^{s} = \alpha_{i}^{s}A_{i}^{t},\: t,s = 0,...,n;\: i=1,...,h$.\\
Then $A^{1}_{i} =\alpha^{1}_{i}\frac{A^{0}_{i}}{\alpha^{0}_{i}},..., A^{n}_{i} =\alpha^{n}_{i}\frac{A^{n}_{i}}{\alpha^{n}_{i}}$. Define $\lambda_{i} = \frac{A_{i}^{0}}{\alpha_{i}^{0}} = ... = \frac{A_{i}^{n}}{\alpha_{i}^{n}}$ for any $i = 1,...,h$. Substituting in $L_{i}^{d-2} = (A^{0}_{i}x_{0}+...+A^{n}_{i}x_{n})^{d-2}$ we get
$$L_{i} = \lambda_{i}^{d-2}(\alpha_{i}^{0}x_{0}+...+\alpha_{i}^{n}x_{n})^{d-2}, \quad i=1,...,h.$$
Then the expressions for the partial derivatives become
$$F_{x_{0}} = \sum_{i=1}^{h}\alpha_{i}^{0}\lambda_{i}^{d-1}(\alpha_{i}^{0}x_{0}+...+\alpha_{i}^{n}x_{n})^{d-1},...,F_{x_{n}} = \sum_{i=1}^{h}\alpha_{i}^{n}\lambda_{i}^{d-1}(\alpha_{i}^{0}x_{0}+...+\alpha_{i}^{n}x_{n})^{d-1}.$$
To lift the decomposition to $F$ consider the Euler formula $F = \sum_{i=1}^{n}x_{i}F_{x_{i}}$. Substituting the above expressions for the partial derivatives and by straightforward computations we get $F = \sum_{i=1}^{h}\frac{1}{\lambda_{i}}L_{i}^{d}$.
\end{proof}

\begin{Proposition}\label{PEul_gen}
Let $F\in K[x_{0},...,x_{n}]_{d}$ be a homogeneous polynomial. Suppose that its partial derivatives of order $s$ admit a simultaneous decomposition in $h$ powers of linear forms $L_{1},\dots,L_h$ such that $L_{1}^{d-s-1},...,L_{h}^{d-s-1}$ are independent in $K[x_{0},...,x_{n}]_{d-s-1}$. Then the decomposition lifts to a decomposition of the polynomial $F$.
\end{Proposition}
\begin{proof}
It is enough to apply Lemma \ref{PEul} recursively. 
\end{proof}

\begin{Alg}\label{alg_der}
\textbf{Input:} A polynomial $F\in K[x_0,\dots,x_n]_d$. 
\begin{itemize}
\item[-] Compute the spaces $H_{\partial F}^s$. If for all $s\geq 1$ we have that
\begin{itemize}
\item[-] either there does not exist a subset of point of $H_{\partial F}^s\cap \mathcal{V}_{d-s}^n$ generating $H_{\partial F}^s$ or;
\item[-] such a subset $\{L_1^{d-s},\dots,L_{h}^{d-s}\}\subset H_{\partial F}^s\cap \mathcal{V}_{d-s}^n$ exists but the powers $L_1^{d-s-1},\dots,L_{h}^{d-s-1}\in K[x_{0},...,x_{n}]_{d-s-1}$ are linearly dependent; 
\end{itemize}
the algorithm fails.
\item[-] Otherwise, take an integer $s$ such that there exists $\{L_1^{d-s},\dots,L_{h}^{d-s}\}\subset H_{\partial F}^s\cap \mathcal{V}_{d-s}^n$ with $L_1^{d-s-1},\dots,L_{h}^{d-s-1}\in K[x_{0},...,x_{n}]_{d-s-1}$ linearly independent.
\item[-] Solve the linear system $F = \sum_{i=1}^h \lambda_iL_i^d$ in the unknowns $\lambda_i$. 
\end{itemize}
\end{Alg}

\begin{Remark}
The advantage of Algorithm \ref{alg_der} is that it does no require to know in advance that $F$ admits a decomposition in $h$ powers. 
\end{Remark}

\section{Mixed tensors}\label{Mix}
In this section we extend the methods developed in Section \ref{GCM} to the non symmetric case.

\begin{Lemma}\label{falt_s_l}
Let $T\in \Sym^{d_1}V_1\otimes\dots\otimes \Sym^{d_p}V_p$ be a mixed tensor admitting a decomposition of the following form
$$T = L_{1,1}^{d_1}\otimes\dots\otimes L_{1,p}^{d_p}+\dots +L_{h,1}^{d_1}\otimes\dots\otimes L_{h,p}^{d_p}$$
where $L_{j,i}\in V_i$ for $i = 1,\dots,p$, $j = 1,\dots,h$. Consider a flattening 
$\widetilde{T}:V_A^{*}\rightarrow V_B$, where $V_A = \Sym^{a_1}V_{1}\otimes\dots\otimes \Sym^{a_p}V_{p}$ and $V_B = \Sym^{b_1}V_{1}\otimes\dots\otimes \Sym^{b_p}V_{b}$. Let $\{e_0^i,\dots,e_{n_{i}}^i\}$ be a basis of $V_{i}$, and $\xi = (\xi_0,\dots,\xi_{M_{A}})\in\mathbb{P}^{M_A}$ with $M_A = \prod_{i = 1}^p\binom{a_i+n_i}{n_i}-1$. Then the linear combination 
$$\xi_0\widetilde{T}((e_0^1)^{a_1}\dots (e_0^p)^{a_p})+\dots +\xi_{M_A}\widetilde{T}((e_{n_1}^1)^{a_1}\dots (e_{n_p}^p)^{a_p})$$ 
is a scalar multiple of  
$$
(L_{1,1}^{a_1}\otimes\dots \otimes L_{1,p}^{a_p})(\xi)(L_{1,1}^{b_1}\otimes \dots \otimes L_{1,p}^{b_p}) +\dots +(L_{h,1}^{a_1}\otimes\dots \otimes L_{h,p}^{a_p})(\xi)(L_{h,1}^{b_1}\otimes \dots \otimes L_{h,p}^{b_p}) 
$$
\end{Lemma}
\begin{proof}
Since the matrix representing the flattening $\widetilde{T}$, in the given bases, is made of blocks which are catalecticant matrices with respect to the symmetric parts of $T$ the claim follows from Lemma \ref{dersp}.
\end{proof}

\begin{thm}\label{gen_S}
Let $T\in \Sym^{d_1}V_1\otimes\dots\otimes \Sym^{d_p}V_p$ be a mixed tensor admitting a decomposition of the form
$$T = L_{1,1}^{d_1}\otimes\dots\otimes L_{1,p}^{d_p}+\dots +L_{h,1}^{d_1}\otimes\dots\otimes L_{h,p}^{d_p}$$
where $L_{j,i}\in V_i$ for $i = 1,\dots,p$, $j = 1,\dots,h$. If there exists a flattening 
$\widetilde{T}:V_A^{*}\rightarrow V_B$, where $V_A = \Sym^{a_1}V_{1}\otimes\dots\otimes \Sym^{a_p}V_{p}$ and $V_B = \Sym^{b_1}V_{1}\otimes\dots\otimes \Sym^{b_p}V_{b}$, such that 
\begin{itemize}
\item[(i)] $\mathbb{P}(\widetilde{T}(V_A^{*}))\cap\Sec_{h-\overline{A}}(\mathcal{SV}^{\underline{n}}_{\underline{b}})$ consists of $\binom{h}{\overline{A}}$ distinct points, and;
\item[(ii)] $\mathbb{P}(\widetilde{T}(V_A^{*}))\cap\Sec_{h-\overline{A}-1}(\mathcal{SV}^{\underline{n}}_{\underline{b}})$ consists of less than $\binom{h-1}{\overline{A}}$ distinct points,
\end{itemize} 
where $\overline{A} = \prod_{i = 1}^p\binom{a_i+n_i}{n_i}-1$ then $T$ has rank $h$ and it is $h$-identifiable. Furthermore, for $h\leq \overline{A}+2$ the identifiability criterion is effective. 
\end{thm}
\begin{proof} 
By Lemma \ref{falt_s_l}, if $\overline{A}$ of the forms $L_{i,1}^{a_1}\otimes\dots \otimes L_{i,p}^{a_p}$ vanish at a point $\xi\in\mathbb{P}^{\overline{A}}$ then 
$$\xi_0\widetilde{T}((e_0^1)^{a_1}\dots (e_0^p)^{a_p})+\dots +\xi_{M_A}\widetilde{T}((e_{n_1}^1)^{a_1}\dots (e_{n_p}^p)^{a_p})\in\Sec_{h-\overline{A}}(\mathcal{SV}^{\underline{n}}_{\underline{b}})$$ 

This determines $\binom{h}{\overline{A}}$ points in the intersection $\mathbb{P}(\widetilde{T}(V_A^{*}))\cap\Sec_{h-\overline{A}}(\mathcal{SV}^{\underline{n}}_{\underline{b}})$.

Now, if $T$ is a linear combination of $h-1$ elementary tensors the same argument will determine $\binom{h-1}{\overline{A}}$ points in $\mathbb{P}(\widetilde{T}(V_A^{*}))\cap\Sec_{h-\overline{A}-1}(\mathcal{SV}^{\underline{n}}_{\underline{b}})$, and this contradicts (ii).

Furthermore, if $T$ admits two different decompositions the argument above shows that we would have more than $\binom{h}{\overline{A}}$ points in $\mathbb{P}(\widetilde{T}(V_A^{*}))\cap\Sec_{h-\overline{A}}(\mathcal{SV}^{\underline{n}}_{\underline{b}})$, and this would contradict (i).

Finally, for the claim on the effectiveness for $h\leq \overline{A}+2$ it is enough to argue as in last part of the proof of Theorem \ref{gen2}.
\end{proof}

\begin{Alg}\label{alg_S}
\textbf{Input:} $T\in \Sym^{d_1}V_1\otimes\dots\otimes \Sym^{d_p}V_p$ admitting a decomposition in $h$ elementary tensors.
\begin{itemize}
\item[-] If the hypotheses of Theorem \ref{gen_S} are not satisfied for all $(A,B)$-flattenings then the method fails. 
\item[-] Otherwise the points 
$$\xi_1,\dots,\xi_{\binom{h}{\overline{A}}}\in \mathbb{P}(\widetilde{T}(V_A^{*}))\cap\Sec_{h-\overline{A}}(\mathcal{SV}^{\underline{n}}_{\underline{b}})$$
are the points where $\overline{A}$ of the $h$ forms $L_{i,1}^{a_1}\otimes\dots \otimes L_{i,p}^{a_p}$ vanish. Note that on each hyperplane $H_i = \{L_{i,1}^{a_1}\otimes\dots \otimes L_{i,p}^{a_p} = 0\}$ there are $\binom{h-1}{\overline{A}-1}$ of the $\xi_i$. 
\item[-] Among all the sets of $\binom{h-1}{\overline{A}-1}$ of the $\xi_i$ compute those spanning a hyperplane. 
\item[-] By Lemma \ref{comb} these sets are exactly $h$ and the $h$ hyperplanes spanned by them are the zero loci of the forms $L_{i,1}^{a_1}\otimes\dots \otimes L_{i,p}^{a_p}$. 
\item[-] Solve the linear system $T = \sum_{i=1}^h\lambda_i(L_{i,1}^{a_1}\otimes\dots \otimes L_{i,p}^{a_p}w)$ in the unknowns $\lambda_i\in K$.
\end{itemize}
\end{Alg} 

\begin{Remark}
Since the expected dimension of $\Sec_{h-\overline{A}}(\mathcal{SV}^{\underline{n}}_{\underline{b}})$ is $\sum_{i=1}^p n_{b_i}(h-\overline{A})+h-\overline{A}-1$ in order to apply Theorem \ref{gen_S} in the Segre-Veronese case we must have $\sum_{i=1}^p n_{i}(h-\overline{A})+h-\overline{A}-1+\prod_{i=1}^p\binom{a_i+n_i}{n_i} < \prod_{i=1}^p\binom{b_i+n_i}{n_i}$ that is 
$$h < \frac{\prod_{i=1}^p\binom{b_i+n_i}{n_i}+(\prod_{i=1}^p\binom{a_i+n_i}{n_i}-1)\sum_{i=1}^p n_{i}}{\sum_{i=1}^p n_{i}+1}.$$
Similarly, in the Segre case since the expected dimension of $\Sec_{h-\overline{A}}(\mathcal{S}^{\underline{n}}_{\underline{b}})$ is $\sum_{i=1}^{p-s}n_{b_i}(h-\overline{A})+h-\overline{A}-1$ we get
$$h < \frac{\prod_{i=p-s+1}^{p}(n_i+1)+(\prod_{i=1}^{p-s}(n_i+1)-1)\sum_{i=p-s+1}^{p}n_i}{\sum_{i=p-s+1}^{p} n_i+1}.$$
Theorem \ref{gen_S} is particularly useful for Segre products of three factors. In this case the actual codimension of $\Sec_{h-\overline{A}}(\mathcal{S}^{\underline{n}}_{\underline{b}})$ is $(n_2-h+\overline{A}+1)(n_3-h+\overline{A}+1)$. Hence, in order the apply Theorem \ref{gen_S} we need to have 
\stepcounter{thm}
\begin{equation}\label{eq_S}
h < \frac{2n_1+n_2+n_3+2-\sqrt{n_2^2-2n_2n_3+n_3^2+4n_1}}{2}.
\end{equation}
As in Remark \ref{rem_cub} when the right hand side of (\ref{eq_S}) is equal to the degree of $\Sec_{h-n_1}(\mathcal{S}^{\underline{n}}_{\underline{b}})$ the equality in (\ref{eq_S}) is allowed. In this case the classical flattenings method in Proposition \ref{prop2gen} works under the bound $h \leq n_1+1$. For instance, when $n_1 = n_2 = n_3 = n$ the bound in (\ref{eq_S}) becomes $h < 2n-\sqrt{n}+1$ while classical flattenings work for $h\leq n+1$.  
\end{Remark}

\section{Magma scripts}\label{mag_s}
A Magma library which implements our algorithms can be downloaded at the following link:
\begin{center}
\url{https://github.com/alaface/tensors-algorithm}
\end{center} 
In the following we explain the main functions in the library.  The function {\tt PolynomialOfRank} generates a random polynomial of a given rank and the functions {\tt Hilbert}, {\tt Sextic} and {\tt Septic} compute a decomposition of a plane curve of degree five, six and seven respectively in seven, nine and twelve powers. Here as some examples on a finite field and on the field of rational numbers.
\begin{tcolorbox}[
    enhanced, 
    breakable]
{\footnotesize
\begin{verbatim}
> load "library.m";
> F,lis,coef := PolynomialOfRank(2,5,7,Rationals());
> time S,Scoef := Hilbert(F);
Time: 3.490
> S;
[
    -2*x[1] - 9/4*x[2] + x[3],
    -3/50*x[1] + x[2] + x[3],
    -2/7*x[1] + x[2] + x[3],
    1/2*x[1] - 9/40*x[2] + x[3],
    27/14*x[1] + 12/7*x[2] + x[3],
    40/27*x[1] - 1/9*x[2] + x[3],
    20/9*x[1] + 5/2*x[2] + x[3]
]
> G := &+[S[j]^5*Scoef[j]:j in [1..#Scoef]];
> F eq G;
true
> time S,Scoef := Sextic(PolynomialOfRank(2,6,9,Rationals()));
Time: 19.300
> time S,Scoef := Septic(PolynomialOfRank(2,7,12,GF(32003)));
Time: 1.750
> P2<[x]> := ProjectiveSpace(RationalField(),2);
> P := x[1]^7+x[2]^7+x[3]^7+(x[1]+x[2]+x[3])^7+(x[1]+2*x[2]+3*x[3])^7 +(x[1]+7*x[2]+5*x[3])^7
   +(x[1]+(1/2)*x[2]+(1/3)*x[3])^7+(x[1]+(1/5)*x[2]+(2/3)*x[3])^7+(x[1]+(1/7)*x[2]+(1/4)*x[3])^7
   +(x[1]+8*x[2]+x[3])^7+(x[1]+(1/11)*x[2]+5*x[3])^7+(x[1]+(3/2)*x[2]+(5/7)*x[3])^7;
> time Septic(P);
[
    x[3],
    x[2],
    x[1],
    x[1] + x[2] + x[3],
    1/3*x[1] + 2/3*x[2] + x[3],
    1/5*x[1] + 7/5*x[2] + x[3],
    x[1] + 8*x[2] + x[3],
    3*x[1] + 3/2*x[2] + x[3],
    1/5*x[1] + 1/55*x[2] + x[3],
    7/5*x[1] + 21/10*x[2] + x[3],
    3/2*x[1] + 3/10*x[2] + x[3],
    4*x[1] + 4/7*x[2] + x[3]
]
[ 1, 1, 1, 1, 2187, 78125, 1, 1/2187, 78125, 78125/823543, 128/2187, 1/16384 ]
Time: 52937.280


\end{verbatim}
}
\end{tcolorbox}
In general all these functions work faster on a finite field than on the field of rational numbers. This difference is particularly appreciable for the function {\tt Septic}. 

The function {\tt TensorOfRank} generates a random mixed tensor. The functions {\tt IsIdentifiable} and {\tt IdentifyForms} are based on Algorithm \ref{alg_S}. The first determines if a tensor is identifiable, while the second actually computes the linear forms in the decomposition of an identifiable tensor. In the following example we consider a homogeneous polynomial of degree three in four variables of rank five. 
\begin{tcolorbox}[
    enhanced, 
    breakable]
{\footnotesize
\begin{verbatim}
> load "library.m";
> Q := Rationals();
>  T,lis,coef,f := TensorOfRank([3],[3],5,Q);
> time IsIdentifiable([3],[3],[1],T,5,Q);
true
Time: 0.110
> time IdentifyTensor([3],[3],T,5,Q);    
[
    [
        x[1] - 21/10*x[2] + 2*x[3] + 2*x[4]
    ],
    [
        x[1] + 7/10*x[2] + 1/10*x[3] + 21/20*x[4]
    ],
    [
        x[1] - 21/2*x[2] - 3/7*x[3] - 27*x[4]
    ],
    [
        x[1] + 5/6*x[2] + 5/21*x[3] + 8/9*x[4]
    ],
    [
        x[1] + x[2] - 3/2*x[3] - 7/9*x[4]
    ]
]
[ -25/18, 1000/343, 1/36, -27/2, 1 ]
Time: 1.254
\end{verbatim}
}
\end{tcolorbox}
Next, we consider a tensor of rank five in $K^{4}\otimes K^{4}\otimes K^{4}$.
\begin{tcolorbox}[
    enhanced, 
    breakable]
{\footnotesize
\begin{verbatim}
> load "library.m";
> Q := Rationals();
> T,lis,coef,f := TensorOfRank([3,3,3],[1,1,1],5,Q);
> time IsIdentifiable([3,3,3],[1,1,1],[1,0,0],T,5,Q);
true
Time: 0.200
> time IdentifyTensor([3,3,3],[1,1,1],T,5,Q);
[
    [
        x[1] + 2*x[2] - 16/7*x[3] - 18/5*x[4],
        x[5] + 18*x[6] - 7/5*x[7] + 10/3*x[8],
        x[9] + 5/4*x[10] + 3*x[11] - 35/4*x[12]
    ],
    [
        x[1] - 12/5*x[2] - 1/15*x[3] - 3/8*x[4],
        x[5] + 12/5*x[6] + x[7] + 24*x[8],
        x[9] + 24/35*x[10] - 27/20*x[11] + 3/40*x[12]
    ],
    [
        x[1] - 2/3*x[2] - 2/7*x[3] + 1/30*x[4],
        x[5] - 5*x[6] + 9*x[7] + 5/3*x[8],
        x[9] - 14/27*x[10] - 7/9*x[11] + 7/9*x[12]
    ],
    [
        x[1] - 3/4*x[2] - 3/4*x[3] - 9/8*x[4],
        x[5] + 6/7*x[6] + 5/9*x[7] - 16/9*x[8],
        x[9] - 9*x[10] - 1/2*x[11] - 4/5*x[12]
    ],
    [
        x[1] + 15/8*x[2] - 35/36*x[3] + 1/12*x[4],
        x[5] + 5/6*x[6] - 21/2*x[7] + 12*x[8],
        x[9] + 20/21*x[10] + 1/15*x[11] + 2/9*x[12]
    ]
]
[ -1/2, 25/27, -27/70, -10/3, 6/5 ]
Time: 1.507
\end{verbatim}
}
\end{tcolorbox}
Finally, we give an example for a tensor of rank six in $K^{5}\otimes \Sym^2K^{5}$.
\begin{tcolorbox}[
    enhanced, 
    breakable]
{\footnotesize
\begin{verbatim}
> load "library.m";
> Q := Rationals();
> T,lis,coef,f := TensorOfRank([4,4],[1,2],6,Q);     
> time IsIdentifiable([4,4],[1,2],[1,0],T,6,Q);      
true
Time: 0.360
> time IdentifyTensor([4,4],[1,2],T,6,Q);    
[
    [
        x[1] - 45/8*x[2] - 10*x[3] + 5*x[4] - 45/2*x[5],
        x[6] + 5/9*x[7] - 25/9*x[8] - 10/9*x[9] + 25/63*x[10]
    ],
    [
        x[1] - 10/9*x[2] - 10/3*x[3] - 10/3*x[4] - 6*x[5],
        x[6] + 10/7*x[7] + 50/63*x[8] - 40/9*x[9] + 5/9*x[10]
    ],
    [
        x[1] - 1/2*x[2] + 1/5*x[3] - 8/45*x[4] - 3/5*x[5],
        x[6] + 2/5*x[7] + 9/40*x[8] + 1/25*x[9] - 8/5*x[10]
    ],
    [
        x[1] + 1/2*x[2] - 2/5*x[3] - 1/3*x[4] + 7/4*x[5],
        x[6] + 3/4*x[7] + 9/8*x[8] - 3/4*x[9] - 3/5*x[10]
    ],
    [
        x[1] + 9/20*x[2] + 27/70*x[3] - 81/80*x[4] + 3/5*x[5],
        x[6] + 10/27*x[7] + 2/9*x[8] - 2/15*x[9] + 7/3*x[10]
    ],
    [
        x[1] - 8/5*x[2] - 36/35*x[3] - 8/7*x[4] - 2/15*x[5],
        x[6] + 2/5*x[7] + 16/35*x[8] + 4/15*x[9] - 8/25*x[10]
    ]
]
[ 54/125, 243/500, 125/2, 2, -5/2, -625/144 ]
Time: 2.073
\end{verbatim}
}
\end{tcolorbox}

\bibliographystyle{amsalpha}
\bibliography{Biblio}

\providecommand{\bysame}{\leavevmode\hbox to3em{\hrulefill}\thinspace}
\providecommand{\MR}{\relax\ifhmode\unskip\space\fi MR }
\providecommand{\MRhref}[2]{%
  \href{http://www.ams.org/mathscinet-getitem?mr=#1}{#2}
}
\providecommand{\href}[2]{#2}
\begin{thebibliography}{BCMT10}

\bibitem[AC20]{AC20}
E.~Angelini and L.~Chiantini, \emph{On the identifiability of ternary forms},
  Linear Algebra Appl. \textbf{599} (2020), 36--65. \MR{4083806}

\bibitem[AH95]{AH95}
J.~Alexander and A.~Hirschowitz, \emph{Polynomial interpolation in several
  variables}, J. Algebraic Geom. \textbf{4} (1995), no.~2, 201--222.
  \MR{1311347}

\bibitem[Bal19]{Ba19}
E.~Ballico, \emph{An effective criterion for the additive decompositions of
  forms}, Rend. Istit. Mat. Univ. Trieste \textbf{51} (2019), 1--12.
  \MR{4048830}

\bibitem[BB12]{BB12}
E.~Ballico and A.~Bernardi, \emph{Decomposition of homogeneous polynomials with
  low rank}, Math. Z. \textbf{271} (2012), no.~3-4, 1141--1149. \MR{2945601}

\bibitem[BCMT10]{BCMT10}
J.~Brachat, P.~Comon, B.~Mourrain, and E.~Tsigaridas, \emph{Symmetric tensor
  decomposition}, Linear Algebra Appl. \textbf{433} (2010), no.~11-12,
  1851--1872. \MR{2736103}

\bibitem[BCP97]{Magma97}
W.~Bosma, J.~Cannon, and C.~Playoust, \emph{The {M}agma algebra system. {I}.
  {T}he user language}, J. Symbolic Comput. \textbf{24} (1997), no.~3-4,
  235--265, Computational algebra and number theory (London, 1993).
  \MR{MR1484478}

\bibitem[BK09]{BK09}
W.~B. Brett and T.~G. Kolda, \emph{Tensor decompositions and applications},
  SIAM Rev. \textbf{51} (2009), no.~3, 455--500. \MR{2535056}

\bibitem[BT20]{BT20}
A.~Bernardi and D.~Taufer, \emph{Waring, tangential and cactus decompositions},
  J. Math. Pures Appl. (9) \textbf{143} (2020), 1--30. \MR{4163122}

\bibitem[CC02]{CC02}
L.~Chiantini and C.~Ciliberto, \emph{Weakly defective varieties}, Trans. Amer.
  Math. Soc. \textbf{354} (2002), no.~1, 151--178. \MR{1859030}

\bibitem[CGLM08]{CGLM08}
P.~Comon, G.~Golub, L-H Lim, and B.~Mourrain, \emph{Symmetric tensors and
  symmetric tensor rank}, SIAM J. Matrix Anal. Appl. \textbf{30} (2008), no.~3,
  1254--1279. \MR{2447451}

\bibitem[CGT14]{CGV14}
E.~Carlini, E.~Guardo, and A.~Van Tuyl, \emph{Star configurations on generic
  hypersurfaces}, J. Algebra \textbf{407} (2014), 1--20. \MR{3197149}

\bibitem[CGT15]{CGV15}
\bysame, \emph{Plane curves containing a star configuration}, J. Pure Appl.
  Algebra \textbf{219} (2015), no.~8, 3495--3505. \MR{3320232}

\bibitem[CM96]{CM96}
P.~Comon and B.~Mourrain, \emph{Decomposition of quantics in sums of powers of
  linear forms}, Signal Processing \textbf{53} (1996), no.~2, 93--107.

\bibitem[CO21]{CO21}
L.~Chiantini and G.~Ottaviani, \emph{A footnote to a footnote to a paper of
  {B}. {S}egre}, https://arxiv.org/abs/2103.04659, 2021.

\bibitem[COV17a]{COV17}
L.~Chiantini, G.~Ottaviani, and N.~Vannieuwenhoven, \emph{Effective criteria
  for specific identifiability of tensors and forms}, SIAM J. Matrix Anal.
  Appl. \textbf{38} (2017), no.~2, 656--681. \MR{3666774}

\bibitem[COV17b]{COV18}
\bysame, \emph{On generic identifiability of symmetric tensors of subgeneric
  rank}, Trans. Amer. Math. Soc. \textbf{369} (2017), no.~6, 4021--4042.
  \MR{3624400}

\bibitem[CT11]{AV11}
E.~Carlini and A.~Van Tuyl, \emph{Star configuration points and generic plane
  curves}, Proc. Amer. Math. Soc. \textbf{139} (2011), no.~12, 4181--4192.
  \MR{2823063}

\bibitem[Dix07]{Di07}
A.~C. Dixon, \emph{The {C}anonical {F}orms of the {T}ernary {S}extic and
  {Q}uaternary {Q}uartic}, Proc. London Math. Soc. (2) \textbf{4} (1907),
  223--227. \MR{1576088}

\bibitem[Dol04]{Do04}
I.~V. Dolgachev, \emph{Dual homogeneous forms and varieties of power sums},
  Milan J. Math. \textbf{72} (2004), 163--187. \MR{2099131}

\bibitem[Fal83]{Fa83}
G.~Faltings, \emph{Endlichkeitss\"{a}tze f\"{u}r abelsche {V}ariet\"{a}ten
  \"{u}ber {Z}ahlk\"{o}rpern}, Invent. Math. \textbf{73} (1983), no.~3,
  349--366. \MR{718935}

\bibitem[GHM13]{GHM13}
A.~V. Geramita, B.~Harbourne, and J.~Migliore, \emph{Star configurations in
  {$\Bbb{P}^n$}}, J. Algebra \textbf{376} (2013), 279--299. \MR{3003727}

\bibitem[GM19]{GM19}
F.~Galuppi and M.~Mella, \emph{Identifiability of homogeneous polynomials and
  {C}remona transformations}, J. Reine Angew. Math. \textbf{757} (2019),
  279--308. \MR{4036576}

\bibitem[Hil88]{Hi88}
D.~Hilbert, \emph{Lettre adressée à m. hermite}, Journal de Mathématiques
  Pures et Appliquées (1888), 249--256 (fre).

\bibitem[IK99]{IK99}
A.~Iarrobino and V.~Kanev, \emph{Power sums, {G}orenstein algebras, and
  determinantal loci}, Lecture Notes in Mathematics, vol. 1721,
  Springer-Verlag, Berlin, 1999, Appendix C by Iarrobino and Steven L. Kleiman.
  \MR{1735271}

\bibitem[LO13]{LO13}
J.~M. Landsberg and G.~Ottaviani, \emph{Equations for secant varieties of
  {V}eronese and other varieties}, Ann. Mat. Pura Appl. (4) \textbf{192}
  (2013), no.~4, 569--606. \MR{3081636}

\bibitem[LO15]{LO15}
\bysame, \emph{New lower bounds for the border rank of matrix multiplication},
  Theory Comput. \textbf{11} (2015), 285--298. \MR{3376667}

\bibitem[MMS18]{MMS18}
A.~Massarenti, M.~Mella, and G.~Staglian\`o, \emph{Effective identifiability
  criteria for tensors and polynomials}, J. Symbolic Comput. \textbf{87}
  (2018), 227--237. \MR{3744347}

\bibitem[MO20]{MO20}
B.~Mourrain and A.~Oneto, \emph{On minimal decompositions of low rank symmetric
  tensors}, Linear Algebra Appl. \textbf{607} (2020), 347--377. \MR{4154967}

\bibitem[MR13]{MR13}
A.~Massarenti and E.~Raviolo, \emph{The rank of {$n\times n$} matrix
  multiplication is at least {$3n^2-2\sqrt{2}n^\frac{3}{2}-3n$}}, Linear
  Algebra Appl. \textbf{438} (2013), no.~11, 4500--4509. \MR{3034546}

\bibitem[OO13]{OO13}
L.~Oeding and G.~Ottaviani, \emph{Eigenvectors of tensors and algorithms for
  {W}aring decomposition}, J. Symbolic Comput. \textbf{54} (2013), 9--35.
  \MR{3032635}

\bibitem[Syl04]{Sy04}
J.~J. Sylvester, \emph{The collected mathematical papers}, vol.~1, Cambridge
  University Press, 1904.

\end{thebibliography}

\end{document}